\documentclass{article}
\usepackage{fontenc}
\usepackage{amsmath}
\usepackage{amssymb}
\usepackage{amsfonts}
\usepackage{amsthm}

\newtheorem{thm}{Theorem}[section]
\newtheorem{prop}[thm]{Proposition}
\newtheorem{cor}[thm]{Corollary}

\newtheorem{lem}[thm]{Lemma}
\newtheorem{obs}[thm]{Observation}

\begin{document}

\title{A $p$-adic Approach to the Weil Representation of Discriminant Forms
Arising from Even Lattices}

\author{Shaul Zemel\thanks{The initial stage of this research has been carried
out as part of my Ph.D. thesis work at the Hebrew University of Jerusalem,
Israel. The final stage of this work was supported by the Minerva Fellowship
(Max-Planck-Gesellschaft).}}

\maketitle

\section*{Introduction}

Let $M$ be an even lattice with dual $M^{*}$ and level $N$. The group
$Mp_{2}(\mathbb{Z})$, the unique non-trivial double cover of
$SL_{2}(\mathbb{Z})$, admits a representation $\rho_{M}$, called the \emph{Weil
Representation}, on the space $\mathbb{C}[M^{*}/M]$. This representation arises
naturally in the theory of theta functions, since a neat description of the
theta function of the lattice $M$ is given by a $\mathbb{C}[M^{*}/M]$-valued
function. This theta function is a modular form with representation $\rho_{M}$
(see, for example, Theorem 4.1 of \cite{[B1]}). In fact, the first papers
dealing with Weil representations (\cite{[Scho]}, \cite{[Kl]}, and others)
studied the modularity of theta functions even before the introduction of the
abstract, general Weil representation in \cite{[W]}. The Weil representations
now play an important role in various branches of mathematics. One example is the representation theory of $SL_{2}(\mathbb{Z}_{p})$ and $GL_{2}(\mathbb{Z}_{p})$, where they include all the irreducible continuous representations of the former group and most of those of the latter (see \cite{[No]} and \cite{[NW2]}). In addition, Examples of important Number-Theoretic applications of the modular behavior of theta functions with respect to the Weil representation can be found in \cite{[B2]} and \cite{[Z2]}.

Several properties of the Weil representation have been known for a long time.
For example, the fact the Weil representation factors through a finite quotient
which is a double cover of $SL_{2}(\mathbb{Z}/N\mathbb{Z})$ is already given, in
a different presentation, in \cite{[Scho]}. In fact, every irreducible
representation of $SL_{2}(\mathbb{Z})$ which factors through such a congruence
quotient is contained in a Weil representation (this follows from the results
of \cite{[No]} and \cite{[NW2]} mentioned above). Moreover, the seminal paper
\cite{[W]}, which initiated the much more general theory of Weil
representations, provides, up to some roots of unity, formulae for the
representation of matrices in which $c=0$ or in which $c$ is invertible (see Eq.
(16) of that reference). Several papers have given explicit formulae for the
action of more general matrices of $Mp_{2}(\mathbb{Z})$, under some conditions
(see \cite{[Kl]}, \cite{[NW1]}, and \cite{[Wo]}), where some (like \cite{[Sn]})
give expressions in terms of Gauss sums which are not explicitly evaluated.
Closed and explicit formulae for the action of the general element of
$Mp_{2}(\mathbb{Z})$ via $\rho_{M}$ seem to appear only recently, in
\cite{[Sche]} for the even signature case and \cite{[Str]} for the general case.
Such formulae are important for calculational purposes --- see, e.g., the
applications mentioned in the Introduction to \cite{[Str]}.

In all these works theta functions play an essential role. Indeed, they are used
in \cite{[Scho]} to prove the factoring of the Weil representation through a
finite quotient, as well as in more general works like \cite{[Kl]},
\cite{[NW1]}, and \cite{[W]}. \cite{[B3]} also uses theta functions to prove
assertions about Weil representations. Later, the factoring property is used in
\cite{[Sche]} and in \cite{[Str]} to prove their formulae. The action of
elements of the form $ST^{m}ST^{n}$ is explicitly calculated there, and then one
evaluate the action of a general element of $SL_{2}(\mathbb{Z})$ or
$Mp_{2}(\mathbb{Z})$ by carefully keeping track of the roots of unity appearing
in the calculation. The formula in \cite{[Sn]} is also proved using theta
functions.

\smallskip

The main aim of this paper is to show how the formulae for the $\rho_{M}$-action
of a general element of $Mp_{2}(\mathbb{Z})$ can be obtained by a direct
evaluation, not depending on ``external objects'' such as theta functions. This
may open the subject to generalizations for which these theta functions do not
exist or their properties are not yet well established. An immediate extension
of this method allows one to prove results about Weil representations arising
from odd lattices as well, where the acting group is an index 3 subgroup of
$Mp_{2}(\mathbb{Z})$ (in this case a treatment using theta functions is also
available, though the associated finite group is no longer a finite quadratic
module in the usual sense). The main technical difficulty is that in most
matrices in $SL_{2}(\mathbb{Z})$, the lower left entry is neither 0 nor
invertible in $\mathbb{Z}$, so that the formulae from \cite{[W]} cannot be
applied directly.

Our method thus goes as follows. We begin by decomposing the Weil representation
$\rho_{M}$ into $p$-parts, in which each $p$-part can be seen as subspace of the
Schwartz functions on the $p$-adic vector space $M_{\mathbb{Q}_{p}}$. This idea
was proposed to the author by E. Lapid. Then we consider the Weil representation
of $Mp_{2}(\mathbb{Q}_{p})$ on the space of Schwartz functions on
$M_{\mathbb{Q}_{p}}$, observe that elements lying over $SL_{2}(\mathbb{Z}_{p})$
preserve our finite-dimensional subspace, and see that restricting to
$Mp_{2}(\mathbb{Z})$ just gives the $p$-part of $\rho_{M}$ again. Since any
non-zero element of $\mathbb{Q}_{p}$ is invertible, it seems that all one has to
do is to evaluate the expressions from Eq. (16) of \cite{[W]} explicitly. The
result involves, in most cases, a Gauss sum (which appears globally without its
explicit value in \cite{[Sn]} and others), which we can (and do) evaluate. Note
that even though any matrix in $SL_{2}(\mathbb{Z}_{p})$ can be written as the
product of at most two matrices with invertible lower left entries, reducing
the expressions arising from evaluating such products to a closed formula turns
out not to be shorter or simpler than our calculations.

The evaluation requires, however, one further, non-trivial step. The operators
attained by the Weil representation are not always those appearing in the
formulae from \cite{[W]}, but rather their multiples by certain roots of unity.
For determining these one has to find which pair of elements, lying over a matrix in $SL_{2}(\mathbb{Q}_{p})$, belong to the metaplectic double cover (and not just to the $S^{1}$-cover). \cite{[Ra]} has determined the full metaplectic double cover $Mp(V)$ of the symplectic group $Sp_{\mathbb{F}}(V \times V)$ for every local field of $\mathbb{F}$ characteristic different from 2, but the
coefficients are given in symplectic notation. We adapt the formulae from
\cite{[Ra]} to obtain a neater presentation of the formulae for the
coefficients, using the Weil index of the quadratic form on $V$ and of
associated quadratic forms, for matrices in $SL_{2}(\mathbb{F})$. The references
\cite{[Ku2]} and \cite{[Ge]} show that the metaplectic cover splits over the
ring of integers wherever the residue field has characteristic different from 2,
a fact which simplifies many calculations. Multiplying the expressions mentioned
in the previous paragraph by the roots of unity described here completes the
evaluation of $\rho_{M}(A)$ for arbitrary matrix $A \in Mp_{2}(\mathbb{Z})$.
This reproduces the formulae of \cite{[Sche]} and \cite{[Str]} for the case of
even $M$, as well as the formula of \cite{[Sn]}. The results for odd lattices
seem to be new.

Since in both \cite{[Sche]} and in \cite{[Str]} the general root of unity
appears as the product of ``$p$-adic factors'', apparently one cannot avoid a
(maybe implicit) $p$-adic decomposition. We remark that \cite{[Sche]} and
\cite{[Str]} consider just finite quadratic modules, while the structure which
we consider includes also an even lattice which yields the finite quadratic
module as its discriminant form . However, our results contain those of
\cite{[Sche]} and \cite{[Str]} in complete generality, since it is well-known
that any finite quadratic module is the discriminant form of some even lattice
(see \cite{[N]}, for example). It it important to note that our $p$-adic factors
(in the even lattice case) do \emph{not} coincide with those of \cite{[Sche]}
and \cite{[Str]}. However, their \emph{total product} does give the same result
as in \cite{[Sche]} and \cite{[Str]}.

Apart from being simple and direct, this method has several advantages. One may
be interested in generalizing the formulae from \cite{[Sche]} and \cite{[Str]}
to Weil representations of larger groups. One example is the Weil representation
arising from a lattice over some ring in a number field other than $\mathbb{Q}$.
Such representations appear, for example, in \cite{[Br]}. Another example might
be (covers of) subgroups of the symplectic group $Sp(M \times M)$ which are
larger than $SL_{2}(\mathbb{Z})$. In these cases the properties proved using
theta series are not known, so that having a direct method may turn out more
useful for obtaining such generalizations. Indeed, we have carried out the local
computations for an arbitrary non-archimedean local field, so that after
evaluating certain Gauss sums the first generalization may be a feasible task.
Even though not every discriminant form over a number field arises from a
lattice (see, e.g., \cite{[Boy]}), such an evaluation may give the idea of how
the representation looks like in general. The second generalization may be
obtained, in some cases, by combining our method with the results of
\cite{[Ra]}.

\smallskip

The paper is divided into 7 sections. Section \ref{LatWRep} goes over the basic
definitions of lattices and the corresponding Weil representations. In Section
\ref{Decom} we present the decomposition into $p$-parts, and the identification
with a subspace of $\mathcal{S}(M_{\mathbb{Q}_{p}})$. Section \ref{GausspLat}
describes Jordan decompositions of $p$-adic lattices and evaluates some
generalized quadratic Gauss sums. In Section \ref{Meta} we survey the Weil
representation associated to a vector space over a local field
$\mathbb{F}\neq\mathbb{C}$, and present the lift of \cite{[Ge]} and
\cite{[Ku2]}. Section  \ref{Local} evaluates the operators from \cite{[W]} on
certain Schwartz functions on vector spaces over non-archimedean local fields.
In Section \ref{Main} we obtain our main results for even lattices. Finally, in
Section \ref{OddGen} we state the necessary adjustments for the case of odd
lattices, and discuss some further possible generalizations.

\smallskip

I am deeply indebted to E. Lapid for his proposal to look for a $p$-adic proof
to the factoring of the Weil representation through (a double cover of)
$SL_{2}(\mathbb{Z}/N\mathbb{Z})$, which initiated my work on this paper (and the
corresponding part in my Ph.D. thesis). I would also like to thank my Ph.D.
advisor R. Livn\'{e} and to H. M. Farkas for their help. I also thank J.
Bruinier, N. Scheithauer and F. Str\"{o}mberg for fruitful discussions while
writing this paper and for referring me to \cite{[Ku1]}. Special thanks are to
T. Yang, for referring me to \cite{[Ra]}. I am also grateful for the two
referees, whose remarks have greatly contributed to the presentation of the
results of this paper.

\section{Even Lattices and Weil Representations \label{LatWRep}}

In this section we give the basic definitions of lattices, the real and integral
metaplectic groups, and the Weil representation.

\medskip

Let $R$ be an integral domain, with field of fractions $\mathbb{K}$ of
characteristic different from 2, and group $R^{*}$ of invertible elements. An
\emph{$R$-lattice} is a free $R$-module $M$, of finite rank which we denote
$rk(M)$, endowed with a symmetric non-degenerate bilinear form, denoted
$(\cdot,\cdot):M \times M \to R$. We write $\lambda^{2}$ for $(\lambda,\lambda)$
with $\lambda \in M$, as well as $M_{S}$ for $M\otimes_{R}S$ wherever $S$ is an
integral domain containing $R$. Multiplying the bilinear form on $M$ by some $0
\neq a \in R$ yields again an $R$-lattice, which we denote $M(a)$. The
\emph{dual module} $M^{*}=Hom(M,R)$ is contained in $M_{\mathbb{K}}$, and
contains $M$. The quotient $D_{M}=M^{*}/M$ is a torsion $R$-module of finite
rank called the \emph{discriminant} of $M$. It inherits a non-degenerate
symmetric $\mathbb{K}/R$-valued bilinear form. An $R$-lattice is called just a
\emph{lattice} (resp. a \emph{$p$-adic lattice}) if $R=\mathbb{Z}$ (resp.
$R=\mathbb{Z}_{p}$). If $R$ is the ring of integers in a global field
$\mathbb{K}$ and $v$ is a non-archimedean place of $\mathbb{K}$ with ring of
integers $\mathcal{O}_{v}$ then we shorten $M_{\mathcal{O}_{v}}$ to simply
$M_{v}$. In this case $D_{M}$ is finite, and we denote its cardinality by
$\Delta_{M}$. The definition we made is perhaps not the most general one (i.e.,
many authors allow the module to be projective and not just free), but will be
sufficient for our purposes as the rings we consider in this papers will all be
principal ideal domains.

An $R$-lattice $M$ is called \emph{even} if 2 divides $\lambda^{2}$ for every
$\lambda \in M$. This is always the case if $2 \in R^{*}$. If $M$ is even then
so is $M_{S}$ for every integral domain $S$ containing $R$. This is a local
property: If $\mathbb{K}$ is a global field and $R$ is its ring of integers then
an $R$-lattice $M$ is even if and only if $M_{v}$ is even for every place $v$ of
$\mathbb{K}$ (namely for every $v$ which lies over 2). If $M$ is even then
$q:\lambda\mapsto\frac{\lambda^{2}}{2}$ is an $R$-valued quadratic form on $M$,
which gives rise to a $\mathbb{K}/R$-valued quadratic form on $D_{M}$ (if
$R=\mathbb{Z}$ then this makes $D_{M}$ a finite quadratic module in the language
of \cite{[Str]} and a discriminant form in the language of \cite{[Sche]}). A
lattice which is not even will be called \emph{odd}.

For an $R$-lattice $M$ we call the ideal $N=\big\{a \in
R\big|a\frac{\gamma^{2}}{2} \in R\mathrm{\ for\ all\ }\gamma \in M^{*}\big\}$
the \emph{level} of $M$. Hence $a(\gamma,\delta) \in R$ for $a \in N$ and
$\gamma$ and $\delta$ from $M^{*}$. In the global case the level of $M_{v}$ is
$N\otimes_{R}\mathcal{O}_{v}$ for any non-archimedean place $v$. For
$R=\mathbb{Z}$ or $R=\mathbb{Z}_{p}$ we allow the slight abuse of notation in
which $N$ may also denote a generator of that ideal. An even lattice is
\emph{unimodular} if and only if it has level 1 (see more generally Lemma
\ref{pfin} below), but this statement is false for odd lattices.

For a $\mathbb{Z}$-lattice $M$ we define its \emph{signature} $sgn(M)$ to be the
signature of $M_{\mathbb{R}}$, namely the dimension of a maximal positive
definite subspace of $M_{\mathbb{R}}$ minus the dimension of a maximal negative
definite subspace there. Its image modulo 8 is what is referred to as the
signature of $D_{M}$ in \cite{[Sche]} and \cite{[Str]}. Adopting the notation
$\mathbf{e}(z)=e^{2\pi iz}$ for complex $z$ and denoting the root of unity
$\mathbf{e}\big(\frac{1}{8}\big)$ (which will appear many times in this paper)
by $\zeta_{8}$, we have \emph{Milgram's formula}, which states that
\[\sum_{\gamma \in
D_{M}}\mathbf{e}\bigg(\frac{\gamma^{2}}{2}\bigg)=\zeta_{8}^{sgn(M)}\sqrt{\Delta_
{M}}.\] The group $\big\{z\in\mathbb{C}\big||z|=1\big\}$ will be denoted
$S^{1}$ in this paper.

\medskip

The group $SL_{2}(\mathbb{R})$ admits a non-trivial double cover
$Mp_{2}(\mathbb{R})$, which has several equivalent descriptions. We use here
the description commonly used in the theory of modular forms (of half-integral
weight), and in Section \ref{Meta} we present the realization arising from the
general theory of \cite{[W]} and give the isomorphism between them. We recall
that any element $A=\binom{a\ \ b}{c\ \ d} \in SL_{2}(\mathbb{R})$ acts on the
upper half-plane $\mathcal{H}=\{\tau\in\mathbb{C}|\Im\tau>0\}$ by
$A\tau=\frac{a\tau+b}{c\tau+d}$, with the factor of automorphy
$j(A,\tau)=c\tau+d$. The group $Mp_{2}(\mathbb{R})$ consists of pairs
$(A,\varphi)$ with $A \in SL_{2}(\mathbb{R})$ and $\varphi$ a holomorphic
function on $\mathcal{H}$ satisfying $\varphi(\tau)^{2}=j(A,\tau)$. The
multiplication is defined by
\[(A,\varphi)(B,\psi)=\big(AB,\tau\mapsto\varphi(B\tau)\psi(\tau)\big),\] which
is well-defined by the cocycle condition $j(AB,\tau)=j(A,B\tau)j(B,\tau)$.

We define $Mp_{2}(\mathbb{Z})$ to be the set of elements in $Mp_{2}(\mathbb{R})$
which lie over $SL_{2}(\mathbb{Z})$. This is a double cover of
$SL_{2}(\mathbb{Z})$. The algebraic description of $Mp_{2}(\mathbb{Z})$ is based
on the 3 elements \[T=\left(\left(\begin{array}{cc} 1 & 1 \\ 0 &
1\end{array}\right),1\right),\quad S=\left(\left(\begin{array}{cc} 0 & -1 \\ 1 &
0\end{array}\right),\sqrt{\tau}\right),\quad Z=\left(\left(\begin{array}{cc} -1
& 0 \\ 0 & -1\end{array}\right),i\right)\] of $Mp_{2}(\mathbb{Z})$, where
$\sqrt{\tau}$ in $S$ takes values with positive real and imaginary parts.
The elements $T$ and $S$ generate $Mp_{2}(\mathbb{Z})$, $Z$ is of order 4
and generates the center of $Mp_{2}(\mathbb{Z})$, and the identities
$S^{2}=(ST)^{3}=Z$ hold. Moreover, these are the only relations in
$Mp_{2}(\mathbb{Z})$. We shall use the same notation $T$, $S$, and $Z$ for
the images of these elements in $SL_{2}(\mathbb{Z})$, as well as in
$Mp_{2}(\mathbb{Z}_{p})$ and in $SL_{2}(\mathbb{Z}_{p})$ for any prime $p$,
without risking confusion. $Mp_{2}(\mathbb{R})$ and $Mp_{2}(\mathbb{Z})$ are
non-trivial covers of $SL_{2}(\mathbb{R})$ and $SL_{2}(\mathbb{Z})$
respectively.

\smallskip

Let $G$ be a locally compact Abelian group. The anti-symmetrization of the
pairing between $G$ and its Pontryagin dual $\widehat{G}$ gives a symplectic
structure on $G\times\widehat{G}$. The symplectic group $Sp(G)$ is defined (as
in \cite{[W]}) to be the group of endomorphisms of $G\times\widehat{G}$ which
preserves this symplectic structure, and the general theory of \cite{[W]} now
gives a faithful unitary representation of an $S^{1}$-cover of $Sp(G)$ on the
space $L^{2}(G)$ which leaves the dense subspace $\mathcal{S}(G)$ of Schwartz
functions invariant. We note that elements of $Sp(G)$ can be written as
$2\times2$ matrices, having one coordinate in $End(G)$, one in
$Hom(G,\widehat{G})$, one in $Hom(\widehat{G},G)$, and one in $End(\widehat{G})$
(satisfying the symplectic condition).

Let $f$ be a non-degenerate character of second degree on $G$. This means a map
$f:G \to S^{1}$ such that the map $(x,y)\mapsto\frac{f(x+y)}{f(x)f(y)}$ is a
bi-homomorphism, and the (symmetric) homomorphism $\rho:G\to\widehat{G}$,
$\rho(x):y\mapsto\frac{f(x+y)}{f(x)f(y)}$ attached to this bi-homomorphism is an
isomorphism. We may thus view $Sp(G)$ as contained in $M_{2}\big(End(G)\big)$.
The intersection of $Sp(G)$ with $M_{2}(\mathbb{Z})$ is precisely
$SL_{2}(\mathbb{Z})$. Restriction yields a representation of an $S^{1}$-cover of
$SL_{2}(\mathbb{Z})$. The classical generators $T$ and $S$ in
$SL_{2}(\mathbb{Z})$ can always be lifted to the elements
$T_{f}=\mathbf{t}_{0}(f)$ and $\widetilde{S}_{f}=\mathbf{d}_{0}'(\rho^{-1})$ (in
the notation of \cite{[W]}), and then one lifts Eq. (9) of \cite{[W]} from
$Sp(G)$ to a similar equation in its $S^{1}$-cover containing a factor
$\gamma(f)$. Following \cite{[Ra]} and others, we call this factor the
\emph{Weil index} of $f$. By Theorem 2 of \cite{[W]}, the Weil index $\gamma(f)$
of $f$ appears in the (distribution-theoretic) Fourier transform of $f$, and
both in this Fourier transform and in $\widetilde{S}_{f}$ the modulus of $\rho$
shows up. We normalize the Haar measure on $G$ so that this modulus equals
unity. Now, if we further assume that $f(-x)=f(x)$ for any $x \in G$ (we call
such $f$ \emph{symmetric}), which is equivalent to the statement that $T_{f}$
commutes with the parity operator $\widetilde{S}_{f}^{2}=\mathbf{d}_{0}(-1)$,
then the lifted Eq. (9) can be written as
$(\widetilde{S}_{f}T_{f})^{3}=\gamma(f)\widetilde{S}_{f}^{2}$ (without the
symmetry condition on $f$, the left hand side is a bit more complicated). This
shows that by defining $S_{f}=\overline{\gamma(f)}\widetilde{S}_{f}$ we obtain
the relation $(S_{f}T_{f})^{3}=S_{f}^{2}$, and the square of this common element
$Z_{f}$ is scalar multiplication by $\overline{\gamma(f)}^{4}$. Hence we obtain
a unitary representation of the pre-image of $SL_{2}(\mathbb{Z})$ in the
universal cover of $SL_{2}(\mathbb{R})$, in which the order of $Z_{f}$ is twice
the order of $\gamma(f)^{4}$ in $S^{1}$. Wherever $\gamma(f)^{8}=1$ the
representation is of $Mp_{2}(\mathbb{Z})$, factoring through
$SL_{2}(\mathbb{Z})$ if and only if $\gamma(f)^{4}=1$.

The Weil representation has the following multiplicative property.
\begin{prop}
Let $G$ and $H$ be locally compact Abelian groups with symmetric non-degenerate
characters of second degree $f$ and $g$ respectively, and let $\rho_{G,f}$ and
$\rho_{H,g}$ be the associated Weil representations. Then restricting the Weil
representation $\rho_{G \times H,f+g}$ from $L^{2}(G \times H)$ (or
$\mathcal{S}(G \times H)$) to the tensor product
$L^{2}(G)\otimes_{\mathbb{C}}L^{2}(H)$ (or
$\mathcal{S}(G)\otimes_{\mathbb{C}}\mathcal{S}(H)$) yields just
$\rho_{G,f}\otimes\rho_{H,g}$. In particular $\gamma(f+g)=\gamma(f)\gamma(g)$. A
similar assertion holds for any finite product. \label{prod}
\end{prop}

\begin{proof}
This follows directly from the results of Section 22 of \cite{[W]}.
\end{proof}

Consider now $G=D_{M}$ with the quadratic form $q$, hence with the character of
second degree $f=\mathbf{e} \circ q$. We identify $D_{M}$ with its dual Abelian
group $Hom(D_{M},\mathbb{Q}/\mathbb{Z})$ and with its dual locally compact group
$\widehat{D_{M}}$ via the bilinear form and $\mathbf{e}$. We denote the space
$L^{2}(D_{M})=\mathbb{C}[M^{*}/M]$ by $V_{\rho_{M}}$. The canonical basis
$(e_{\gamma})_{\gamma \in M^{*}/M}$ is $L^{2}$-orthogonal, with all the elements
having the same $L^{2}$-norm $\frac{1}{\sqrt{\Delta_{M}}}$ in our normalization.
Theorem 2 of \cite{[W]} and Milgram's formula imply
$\gamma(f)=\zeta_{8}^{sgn(M)}$. Hence $\gamma(f)^{8}=1$ and the representation
is of $Mp_{2}(\mathbb{Z})$. It is described explicitly by the familiar formulae
appearing in \cite{[B1]}, \cite{[B2]}, and \cite{[Str]}:
\begin{equation}
\begin{array}{rcl} \rho_{M}(T)(e_{\gamma}) & =
& \mathbf{e}(\gamma^{2}/2)e_{\gamma} \\
\rho_{M}(S)(e_{\gamma}) & = &
\displaystyle{\frac{\zeta_{8}^{-sgn(M)}}{\sqrt{\Delta_{M}}}\sum_{\delta \in
M^{*}/M}\mathbf{e}(-(\gamma,\delta))e_{\delta}}.\end{array} \label{rhoMform}
\end{equation}
This is the representation which is the main object of research in this paper.
The condition $\gamma(f)^{4}=1$, for the representation to factor through
$SL_{2}(\mathbb{Z})$, is equivalent to the signature (or equivalently the rank)
of $M$ being even. Note that \cite{[Sche]} considers the dual represenstation
$\rho_{M(-1)}$. Proposition \ref{prod} implies that the representation $\rho_{M
\oplus N}$ associated with an orthogonal direct sum is the tensor product
$\rho_{M}\otimes\rho_{N}$. This resembles the basic idea of the decomposition
into $p$-parts considered below.

\section{Decomposition into $p$-Parts \label{Decom}}

In this section we show how $\rho_{M}$ can be written as the tensor product of
representations on finite dimensional spaces of Schwartz functions on $p$-adic
spaces. We shall ultimately evaluate the action of elements of
$Mp_{2}(\mathbb{Z})$ on these spaces of Schwartz functions, considered as
elements of metaplectic groups over $p$-adic fields.

\medskip

First we introduce some notation. For any prime $p$ there is a natural
isomorphism between $\mathbb{Q}_{p}/\mathbb{Z}_{p}$ and the subgroup
$\mathbb{Z}\big[\frac{1}{p}\big]/\mathbb{Z}$ of $\mathbb{Q}/\mathbb{Z}$.
Composing this embedding with $\mathbf{e}$ yields a character on
$\mathbb{Q}_{p}$, with kernel $\mathbb{Z}_{p}$, which we denote $\chi_{p}$. We
have the equality $\mathbf{e}(x)=\prod_{p}\chi_{p}(x)$ (with almost all factors
being equal to 1) for every
$x\in\mathbb{Q}/\mathbb{Z}=\bigoplus_{p}\mathbb{Z}\big[\frac{1}{p}\big]/\mathbb{
Z}$, hence also for $x\in\mathbb{Q}$. This elementary observation will turn out
to be very useful later.

Given an even lattice $M$, we decompose $D_{M}$ as the direct sum
$\bigoplus_{p}(D_{M})_{p}$, where
$(D_{M})_{p}=D_{M}\otimes_{\mathbb{Z}}\mathbb{Z}_{p}=M_{p}^{*}/M_{p}=D_{M_{p}}$
has cardinality $\Delta_{M_{p}}=p^{v_{p}(\Delta_{M})}$ for every prime $p$. To
each $D_{M_{p}}$ we associate, as in Section \ref{LatWRep}, a Weil
representation $\rho_{M_{p}}$ on the space
$V_{\rho_{M_{p}}}=L^{2}(D_{M_{p}})=\mathbb{C}[M_{p}^{*}/M_{p}]$ using the
character of second degree $f_{p}=\chi_{p} \circ q_{p}$. It may be described
explicitly as in Eq. \eqref{rhoMform} by using the natural basis
$(e_{\gamma_{p}})_{\gamma_{p} \in M_{p}^{*}/M_{p}}$ for $V_{\rho_{M_{p}}}$, but
then the root of unity $\zeta_{8}^{sgn(M)}$ must be replaced by the Weil index
$\gamma(f_{p})$. The fourth power of this Weil index is 1 for odd $p$ and
$(-1)^{rk(M)}$ for $p=2$, so that the representation thus obtained is of
$Mp_{2}(\mathbb{Z})$, factoring through $SL_{2}(\mathbb{Z})$ if $p$ is odd or
$rk(M)$ is even. We now have
\begin{lem}
For any prime $p$, the following are equivalent: $(i)$ $p$ does not divide
$\Delta_{M}$. $(ii)$ $p$ does not divide $N$. $(iii)$ The representation
$\rho_{M_{p}}$ is trivial. \label{pfin}
\end{lem}

\begin{proof}
As $\ker\chi_{p}=\mathbb{Z}_{p}$, the triviality of $\rho_{M_{p}}(T)$ is
equivalent to $N$ being in $\mathbb{Z}_{p}^{*}$, and by non-degeneracy also to
$\Delta_{M_{p}}$ being equal to 1. Since Theorem 5 of \cite{[W]} or Theorem 3 of
\cite{[C]} show that if $\Delta_{M_{p}}=1$ then $\gamma(f_{p})=1$, this
completes the proof of the lemma.
\end{proof}

Lemma \ref{pfin} implies that $\bigotimes_{p}\rho_{M_{p}}$ is well-defined. In
fact, more is true:
\begin{prop}
We have $\rho_{M}=\bigotimes_{p}\rho_{M_{p}}$ as representations of
$Mp_{2}(\mathbb{Z})$. \label{rhoMdecom}
\end{prop}

\begin{proof}
First observe that $V_{\rho_{M}}=\bigotimes_{p}V_{\rho_{M_{p}}}$ (since
$D_{M}=\bigoplus_{p}D_{M_{p}}$), so that $\rho_{M}$ and
$\bigotimes_{p}\rho_{M_{p}}$ act on the same space. Now, as $\gamma \in D_{M}$
equals $\sum_{p}\gamma_{p}$ with $\gamma_{p}$ being the image of $\gamma$ in
$D_{M_{p}}$, we may use the $p$-adic decomposition of $\mathbb{Q}/\mathbb{Z}$
and obtain that $(\gamma,\delta)=\sum_{p}(\gamma_{p},\delta_{p})$ for any
$\gamma$ and $\delta$ in $D_{M}$. Thus the decomposition
$D_{M}=\bigoplus_{p}D_{M_{p}}$ is an \emph{orthogonal} decomposition. As similar
considerations yield also
$\frac{\gamma^{2}}{2}=\sum_{p}\frac{\gamma_{p}^{2}}{2}$
and $f(\gamma)=\prod_{p}f_{p}(\gamma_{p})$, the assertion follows from
Proposition \ref{prod} since the tensor product is essentially finite, i.e.,
only finitely many representations involved in the product are non-trivial. This
proves the proposition.
\end{proof}

\medskip

For any prime $p$, consider now the locally compact group $M_{\mathbb{Q}_{p}}$,
with the character of second degree $f_{\mathbb{Q}_{p}}$ which is the
composition of the quadratic form with $\chi_{p}$. We have
$\gamma(f_{\mathbb{Q}_{p}})=\gamma(f_{p})$ by Section 27 of \cite{[W]} or
Theorem 3 of \cite{[C]}. Since all these $p$-adic vector spaces, as well as the
real vector space $M_{\mathbb{R}}$, arise from one rational vector space
$M_{\mathbb{Q}}$, the Weil indices $\gamma(f_{\mathbb{Q}_{p}})$ and
$\gamma(f_{\mathbb{R}})$ are related through the Weil reciprocity law (for the
global field $\mathbb{Q}$). Its classical formulation is
$\prod_{p\leq\infty}\gamma(f_{\mathbb{Q}_{p}})=1$, with
$\mathbb{Q}_{\infty}=\mathbb{R}$, under the normalization in which
$\prod_{p\leq\infty}\chi_{p}(x)=1$ for every $x\in\mathbb{Q}$. In our
normalization of the characters, in which the character on $\mathbb{R}$ is
$x\mapsto\mathbf{e}(x)$ and coincides with $\prod_{p<\infty}\chi_{p}(x)$ for
$x\in\mathbb{Q}$, the Weil reciprocity law takes the form
$\prod_{p}\gamma(f_{\mathbb{Q}_{p}})=\gamma(f_{\mathbb{R}})=\zeta_{8}^{sgn(M)}$
(the latter equality follows from the evaluations in Section 26 of \cite{[W]}).
Note that the Weil reciprocity law also follows from comparing the global
coefficients in the equality $\rho_{M}(S)=\bigotimes_{p}\rho_{M_{p}}(S)$. Now,
applying the process described in Section \ref{LatWRep} to
$G=M_{\mathbb{Q}_{p}}$ yields an action of an $S^{1}$-cover of the symplectic
group $Sp(M_{\mathbb{Q}_{p}})$ on $L^{2}(M_{\mathbb{Q}_{p}})$ and on the dense
subspace $\mathcal{S}(M_{\mathbb{Q}_{p}})$. Moreover, Section 35 of \cite{[W]}
shows that this representation is continuous in the strong topology on the group
of unitary operators on $L^{2}(M_{\mathbb{Q}_{p}})$, and can be restricted to a
representation of a double cover $Mp(M_{\mathbb{Q}_{p}})$ of
$Sp(M_{\mathbb{Q}_{p}})$. We will be interested in the restriction of the
latter representation to a double cover $Mp_{2}(\mathbb{Q}_{p})$ of
$SL_{2}(\mathbb{Q}_{p})$, and further to a double cover $Mp_{2}(\mathbb{Z}_{p})$
of $SL_{2}(\mathbb{Z}_{p})$. The representation obtained in Section
\ref{LatWRep} for this case, which takes $T$ to $T_{f_{\mathbb{Q}_{p}}}$ and $S$
to $S_{f_{\mathbb{Q}_{p}}}$, is just the restriction of this representation even
further, to $Mp_{2}(\mathbb{Z})$.

The space $\mathcal{S}(M_{\mathbb{Q}_{p}})$ consists of those functions on
$M_{\mathbb{Q}_{p}}$ whose support is contained in some finite
$\mathbb{Z}_{p}$-submodule of $M_{\mathbb{Q}_{p}}$ and which are constant on
cosets of a (smaller) $\mathbb{Z}_{p}$-submodule of $M_{\mathbb{Q}_{p}}$. The
space $V_{\rho_{M_{p}}}$ is naturally isomorphic to the subspace of
$\mathcal{S}(M_{\mathbb{Q}_{p}})$ consisting of those functions which are
supported in $M_{p}^{*}$ and are constant on cosets of $M_{p}$, by identifying
the canonical basis element $e_{\gamma_{p}}$ of $V_{\rho_{M_{p}}}$ with the
characteristic function $E_{M_{p}+\gamma_{p}}$ of the coset $M_{p}+\gamma_{p}$
in $M_{\mathbb{Q}_{p}}$. The crucial point in our method is the following
\begin{obs}
The restriction of the Weil representation arising from $M_{\mathbb{Q}_{p}}$ to
$Mp_{2}(\mathbb{Z})$ preserves this subspace of
$\mathcal{S}(M_{\mathbb{Q}_{p}})$. Moreover, the representation of
$Mp_{2}(\mathbb{Z})$ thus obtained becomes, under our identifications,
precisely $\rho_{M_{p}}$. \label{Zpext}
\end{obs}
Observation \ref{Zpext} follows from Corollary \ref{r0ACDM} below (in which we
take $\mathbb{F}=\mathbb{Q}_{p}$ and $\lambda=\chi_{p}$). However, we remark
that one may prove Observation \ref{Zpext} directly, and indicate how one
applies it to get a simple proof of the fact that $\rho_{M}$ factors through a
double cover of $SL_{2}(\mathbb{Z}/N\mathbb{Z})$. Indeed, Observation
\ref{Zpext} is established by comparing the action of the generators
$T_{f_{\mathbb{Q}_{p}}}$ and $S_{f_{\mathbb{Q}_{p}}}$ with that of $T_{f_{p}}$
and $S_{f_{p}}$
respectively. We then prove the following

\begin{lem}
For every prime number $p$, the minimal closed normal subgroup of
$SL_{2}(\mathbb{Z}_{p})$ containing $T^{N}$ is the group
$\Gamma(N,\mathbb{Z}_{p})$ consisting of those matrices in
$SL_{2}(\mathbb{Z}_{p})$ which are congruent to the identity matrix modulo $N$.
\label{normN}
\end{lem}

\begin{proof}
If $\Gamma$ denotes the minimal normal closed subgroup in question then we
clearly have $\Gamma\leq\Gamma(N,\mathbb{Z}_{p})$ since
$T^{N}\in\Gamma(N,\mathbb{Z}_{p})$. For the reverse inclusion first note that
$T^{r}=\binom{1\ \ r}{0\ \ 1}$ as well as the conjugate $\binom{1\ \ 0}{r\ \ 1}$
of its inverse lie in $\Gamma$ for every $r \in N\mathbb{Z}$, and continuity
allows us to extend the latter assertion to $r \in N\mathbb{Z}_{p}$. Let
$\widetilde{\Gamma}$ be the subgroup of $SL_{2}(\mathbb{Z}_{p})$ consisting of
those matrices $\binom{a\ \ b}{c\ \ d} \in SL_{2}(\mathbb{Z}_{p})$ such that $b
\equiv c\equiv0(\mathrm{mod\ }N)$ and $a \equiv d\equiv1(\mathrm{mod\ }N^{2})$.
Any element of $\binom{a\ \ b}{c\ \ d}\in\widetilde{\Gamma}$ in which
$d\in\mathbb{Z}_{p}^{*}$ (the latter condition is redundant if $p|N$, but not
otherwise) may be now written as  \[\left(\begin{array}{cc} a & b \\ c &
d\end{array}\right)=\left(\begin{array}{cc} 1 & \frac{b-N}{d} \\ 0 &
1\end{array}\right)\left(\begin{array}{cc} 1 & 0 \\ \frac{d-1}{N} &
1\end{array}\right)\left(\begin{array}{cc} 1 & N \\ 0 &
1\end{array}\right)\left(\begin{array}{cc} 1 & 0 \\ \frac{1-d+Nc}{Nd} &
1\end{array}\right)\] (recall that $a=\frac{1+bc}{d}$ in
$SL_{2}(\mathbb{Z}_{p})$), proving that it lies in $\Gamma$. This completes the
proof for the case where $p$ does not divide $N$ (hence
$\widetilde{\Gamma}$ and $\Gamma(N,\mathbb{Z}_{p})$ are both the full group
$SL_{2}(\mathbb{Z}_{p})$), since any matrix in $SL_{2}(\mathbb{Z}_{p})$ is the
product of at most two matrices with invertible lower right entry. Assuming now
$p|N$, we now observe that the matrix $\binom{1+kN\ \ \ \ kN\ \ }{\ -kN\ \ \
1-kN}$ lies in $\Gamma$ for every $k\in\mathbb{Z}_{p}$, as the conjugate of
$\binom{\ 1\ \ \ 0}{kN\ \ 1}$ by $T$. A general element $\binom{a\ \ b}{c\ \
d}\in\Gamma(N,\mathbb{Z}_{p})$ can now be written as \[\left(\begin{array}{cc}
a-kN(a+b) & b+kN(a+b) \\ c-kN(c+d) &
d+kN(c+d)\end{array}\right)\left(\begin{array}{cc} 1+kN & -kN \\ kN &
1-kN\end{array}\right),\] and if we choose $k$ such that
$k\equiv\frac{a-1}{N}\equiv-\frac{d-1}{N}(\mathrm{mod\ }N)$ (the numbers
$\frac{a-1}{N}$ and $-\frac{d-1}{N}$ are congruent modulo $N$ by the
$\Gamma(N,\mathbb{Z}_{p})$ condition) then both factors were seen to lie in
$\Gamma$. This completes the proof of the lemma.
\end{proof}

Lemma \ref{normN} and the triviality of $\rho_{M_{p}}(T^{N})$ now imply the
triviality of $\rho_{M_{p}}$ on all of $\Gamma(N,\mathbb{Z}_{p})$ for any odd
$p$, and the triviality of at least a double cover of
$\Gamma(N,\mathbb{Z}_{2})$ (this is so, since $\rho_{M_{p}}$ factors through
$SL_{2}(\mathbb{Z}_{p})$ for odd $p$ but not necessarily for $p=2$). Since being
in $\Gamma(N)$ is a local property, Lemma \ref{rhoMdecom} completes the
verification of the factoring assertion. However, we shall not use this
assertion in what follows, but rather obtain it again as a special case of the
general formulae.

\section{$p$-adic Lattices and their Gauss Sums \label{GausspLat}}

Many roots of unity which we shall later encounter will be expressed in terms of
Weil indices of $p$-adic Jordan components. In this Section we thus skim through
Jordan decompositions of (even) $p$-adic lattices, together with the
corresponding Gauss sums and Weil indices. Our treatment is related to the
discussion in \cite{[Sche]} and \cite{[Str]} about discriminant forms, the main
difference being the fact that lattices may have unimodular parts (i.e., Jordan
components of the sort $1^{\kappa n}$ or $1^{\kappa n}_{II}$---see the
definitions below), which are no longer visible in their discriminant forms.

\medskip

We begin with some notation. For any odd $K$ we define
$\varepsilon_{K}\in\{1,i\}$ and
$\varepsilon(K)\in\mathbb{F}_{2}=\mathbb{Z}/2\mathbb{Z}$ to be such that
$\varepsilon_{K}^{2}=(-1)^{\varepsilon(K)}=(-1)^{(K-1)/2}$. In addition, define
$\sigma(x)\in\mathbb{F}_{2}$ for non-zero $x$ in $\mathbb{Q}$ (or in
$\mathbb{R}$) such that $sgn(x)=(-1)^{\sigma(x)}$. We extend the Legendre symbol
$\big(\frac{x}{y}\big)$ also for negative odd $y$ by defining
$\big(\frac{x}{y}\big)=\big(\frac{x}{|y|}\big)$. We remark that this is
different from the Kronecker extension used in \cite{[Str]}, \cite{[B3]}, and
\cite{[Sn]}, which in our notation is given by
$\big(\frac{x}{|y|}\big)(-1)^{\sigma(x)\sigma(y)}$. The advantage of our
extension is that $\big(\frac{x}{y}\big)$ depends only on the value of $x$
modulo $y$ also for negative $y$. We also define $\big(\frac{0}{\pm1}\big)$ to
be 1 (in order to preserve the latter property for $y=\pm1$). Moreover, our
convention extends further to the \emph{quadratic power residue symbol} defined
over more general number fields in page 24 of \cite{[Ge]}. Both extensions are
multiplicative in $x$ and in $y$, and in both extensions the quadratic
reciprocity law extend to the formula
$\big(\frac{x}{y}\big)\big(\frac{y}{x}\big)=(-1)^{
\varepsilon(x)\varepsilon(y)+\sigma(x)\sigma(y)}$, holding for every odd $x$ and
$y$ which are coprime. For $x=-1$ we get
$\big(\frac{-1}{y}\big)=(-1)^{\varepsilon(y)+\sigma(y)}$ (in comparison to the
equality $\big(\frac{-1}{y}\big)=(-1)^{\varepsilon(y)}$ holding also for
negative $y$ in the extension from the other references). Note, in relation with
Section \ref{Meta}, that $(-1)^{\sigma(x)\sigma(y)}$ is the Hilbert symbol
$(x,y)_{\mathbb{R}}$.

Let $k\geq0$. The fact that $\big(\frac{2^{k}}{y}\big)$ is defined by the
residue of $y$ modulo 8 (and then $\big(\frac{2^{k}}{y}\big)$ is symmetric in
the sign of $y$) and $\big(\frac{x}{p^{k}}\big)$ (with $p$ odd) is defined by
the residue of $x$ modulo $p$ allows us to extend these particular cases of the
Legendre symbol to $y\in\mathbb{Z}_{2}^{*}$ and $x\in\mathbb{Z}_{p}$
respectively. For $k\geq1$ the latter vanishes for $x \in p\mathbb{Z}_{p}$,
while for $k=0$ it equals 1 for any $x\in\mathbb{Z}_{p}$, invertible or not. We
shall also make use of the following formula, which holds for any odd number
$x$:
\begin{equation}
\bigg(\frac{2}{x}\bigg)\varepsilon_{x}=\zeta_{8}^{1-x}. \label{oddmod8}
\end{equation}
This formula appears as Eq. (5.6) of \cite{[Str]}, and we extend it by
continuity to $x\in\mathbb{Z}_{2}^{*}$. The proof is obtained by checking the 4
possibilities of $x$ modulo 8.

\medskip

It is well-known that any $p$-adic lattice $M$ is isomorphic to an orthogonal
direct sum $\bigoplus_{e=0}^{k}M_{e}(p^{e})$, with $M_{e}$ unimodular for any
$e$ (see, e.g., Proposition 2.6 of \cite{[Z1]} for a much more general
statement, as well as the books and articles cited in that reference). The
sublattices $M_{e}(p^{e})$ are called \emph{Jordan components}, and they are
represented by symbols of the form $q^{\kappa n}$ if $p\neq2$ and $q^{\kappa
n}_{t}$ or $q^{\kappa n}_{II}$ (the latter appears only with even $n$) if $p=2$.
Here $n\in\mathbb{N}$ and $\kappa\in\{\pm\}$, and for $p=2$ the index $t$ lies
in $\mathbb{Z}/8\mathbb{Z}$. Such a symbol with $q=1$ stands for a unimodular
lattice, whose rank is $n$ and whose discriminant satisfies
$\big(\frac{disc(M_{e})}{p}\big)=\kappa$ for odd $p$ and
$\big(\frac{2}{disc(M_{e})}\big)=\kappa$ for $p=2$. These invariants
characterize the unimodular $p$-adic lattice if $p$ is odd. For $p=2$ we
distinguish among even unimodular lattices (which correspond to the subscript
$II$ and are again characterized by $n$ and $\kappa$), and odd unimodular
lattices, for which $t$ is the trace of a diagonal form of $M_{e}$ in
$\mathbb{Z}_{2}/8\mathbb{Z}_{2}=\mathbb{Z}/8\mathbb{Z}$ (this can be seen to be
independent of the diagonal form chosen, and characterize the unimodular lattice
together with $n$ and $\kappa$). For general $q=p^{e}$, the symbols
$q^{\kappa n}$, $q^{\kappa n}_{t}$, and $q^{\kappa n}_{II}$ represent the
lattices obtained by multiplying the bilinear form on $1^{\kappa n}$,
$1^{\kappa n}_{t}$, and $1^{\kappa n}_{II}$ respectively by $q$. If $p=2$ then
the index $t$ must be of the same parity as $n$, and for small values of $n$ not
all the combinations of $t \equiv n(\mathrm{mod\ }2)$ and $\kappa$ can appear:
For $n=1$ we know that $t=\pm1$ implies $\kappa=+$ while $t=\pm5$ implies
$\kappa=-$, while for $n=2$ we have that $t=0$ implies $\kappa=+$ while $t=4$
implies $\kappa=-$. The trivial component, with $n=0$, will always be assumed to
have $\kappa=+$, and index $II$ if $p=2$. For odd $p$ this decomposition is
unique in the sense that direct sums with different invariants are never
isomorphic (this has been shown by many authors; for a recent generalization to
lattices over complete valuation rings of arbitrary rank see \cite{[Z1]}). For
$p=2$ different decomposed forms may give isomorphic 2-adic lattices, but it is
known precisely when this happens (see \cite{[J]}, with some remarks in
\cite{[Z1]}). A $p$-adic lattice is even wherever $p$ is odd or $p=2$ and
$M_{0}$ is of the form $1^{\kappa n}_{II}$ with even $n$.

Any decomposition of $M$ as $\bigoplus_{e=0}^{k}M_{e}(p^{e})$ with $M_{e}$
unimodular for every $e$ is called a \emph{Jordan decomposition}, and the
sublattices $M_{e}(p^{e})$ (or equivalently $q^{\kappa n}$, $q^{\kappa n}_{t}$,
or $q^{\kappa n}_{II}$) are called the \emph{components} of the decomposition,
or, more abstractly, \emph{Jordan components}. In the direct sum of two Jordan
components with the same $q$, the ranks are added and the signs are multiplied.
For $p=2$ the index $t$ is added, $II$ is considered to be 0 when added to some
$t$, and the sum of two $II$ indices remains $II$. A Jordan component is
\emph{indecomposable} if it cannot be presented as the orthogonal direct sum of
smaller $p$-lattices. This is the case only for $q^{\kappa1}$ if $p\neq2$ and
$q^{\kappa1}_{t}$ and $q^{\kappa2}_{II}$ if $p=2$.

\smallskip

We have seen above that if $M$ is a $p$-adic lattice and the characters of
second degree $f$ on $D_{M}$ and $f_{\mathbb{Q}_{p}}$ on $M_{\mathbb{Q}_{p}}$
are defined through composition with $\chi_{p}$ then
$\gamma(f)=\gamma(f_{\mathbb{Q}_{p}})$. We call this common root of unity the
\emph{Weil index of $M$}, and denote it by $\gamma(M)$. It coincides also with
the root of unity denoted $\gamma_{p}$ in \cite{[Sche]} and \cite{[Str]}, which
is given in terms of the elements of $\mathbb{Z}/8\mathbb{Z}$ called $p$-excess
and oddity (or signature) in these references. It is evaluated in the following
\begin{prop}
The Weil index of a $p$-adic Jordan component $q^{\kappa n}$ with odd $p$ is
$\kappa^{v_{p}(q)}\zeta_{8}^{n(1-q)}$. The Weil index of a 2-adic Jordan
component $q^{\kappa n}_{t/II}$ is $\kappa^{v_{2}(q)}\zeta_{8}^{t}$, where for
the index $II$ we take $t=0$. Moreover, for any even $p$-adic lattice $M$ the
equality $\sum_{\eta \in
M^{*}/M}\chi_{p}\big(\frac{\eta^{2}}{2}\big)=\gamma(M)\sqrt{\Delta_{M}}$ holds.
\label{gammap}
\end{prop}

For a proof see Proposition 3.1 of \cite{[Sche]} and the results of Section 3 of
\cite{[Str]}---note that the $\kappa$ factors account for the number $k$ (called
antisquare in \cite{[B3]}) which is defined in these references by
distinguishing different cases. An alternative, simpler proof can be given using
the multiplicativity of all quantities with respect to direct sums and a
$p$-adic analog of Lemma 1 in Appendix 4 of \cite{[MH]}. These two properties
reduce the proof to the verification of the assertions only for indecomposable
even Jordan components of low prime power (namely $1^{\kappa1}$ and
$p^{\kappa1}$ for odd $p$ as well as $1^{\kappa2}_{II}$, $2^{\kappa2}_{II}$,
$2^{\kappa1}_{t}$ and $4^{\kappa1}_{t}$ for $p=2$). The classical result of
Gauss and Eq. \eqref{oddmod8} complete the odd $p$ case, and for $p=2$ the
verification is direct and simple. The oddity formula in \cite{[Sche]} and
\cite{[Str]} is just an incarnation of the Weil reciprocity law.

\medskip

It turns out useful to compare $\gamma(M)$ with $\gamma\big(M(c)\big)$ for some
non-zero $c\in\mathbb{Z}_{p}$. It suffices to restrict our attention to the case
where $M$ is a Jordan component and $c$ is either a power of $p$ or an element
of $\mathbb{Z}_{p}^{*}$. The effect of the former is just changing the power of
$p$ in the symbol of the Jordan component, and Proposition \ref{gammap} implies
that the Weil index depends only on the parity of that power. Furthermore, for
$p=2$ we have $\gamma\big(M(2^{l})\big)=\eta^{l}\gamma(M)$ where $\eta$ is the
total sign of $M$ (i.e., the product of the signs of all the Jordan components).
In particular this term is just a sign. Note, however, that for
$p\equiv3(\mathrm{mod\ }4)$ multiplication by $p$ does not necessarily change
the Weil index only by a sign (if $p\equiv1(\mathrm{mod\ }4)$ then all the Weil
indices lie in $\pm1$).

The action of elements from $\mathbb{Z}_{p}^{*}$ is described in the following
\begin{lem}
Let $M$ be an even $p$-adic lattice and let $a\in\mathbb{Z}_{p}^{*}$. Then
$\gamma\big(M(a)\big)$ equals $\big(\frac{a}{\Delta_{M}}\big)\gamma(M)$ for odd
$p$ and equals $\big(\frac{\Delta_{M}}{a}\big)\gamma(M)^{a}$ for $p=2$.
\label{gammacomp}
\end{lem}

The expression for $p=2$ is well-defined since the exponent $a$ is in fact considered as an element of $\mathbb{Z}_{2}/8\mathbb{Z}_{2}=\mathbb{Z}/8\mathbb{Z}$.

\begin{proof}
It suffices to verify the assertion for the Jordan components. Hence assume that
$M=M_{e}(q)$ with $M_{e}$ unimodular of rank $n$. Thus, $\Delta_{M}=q^{n}$, and
$M(a)=\big(M_{e}(a)\big)(q)$ with $M_{e}(a)$ unimodular. The discriminant of
$M_{e}(a)$ is $a^{n}$ times that of $M_{e}$, so that the sign is multiplied by
$\big(\frac{a}{p}\big)^{n}$ for odd $p$ and by $\big(\frac{2}{a}\big)^{n}$ for
$p=2$. For $p=2$ and odd $M_{e}$ the index $t$ is multiplied by $a$ when
replacing $M_{e}$ by $M_{e}(a)$ (verified using any diagonal form), while an
index $II$ remains unaffected. Hence we have $q^{\kappa n}(a) \cong
q^{(\frac{a}{p})^{n}\kappa n}$ for odd $p$ and $q^{\kappa n}_{t/II}(a) \cong
q^{(\frac{2}{a})^{n}\kappa n}_{at/II}$ if $p=2$. Now apply Proposition
\ref{gammap}, using the value of $\Delta_{M}$ and observing that if $p=2$ then
$a$ is odd and $(\pm1)^{a}=\pm1$. This proves the lemma.
\end{proof}

As already asserted in \cite{[W]}, we also have
$\gamma\big(M(-1)\big)=\overline{\gamma(M)}$ for every $p$-adic lattice $M$.
Combining this with the case $a=-1$ in Lemma \ref{gammacomp} implies that
$\gamma(M)^{2}=\big(\frac{-1}{\Delta_{M_{p}}}\big)$ for odd $p$. Lemma
\ref{gammacomp} and the last two assertions are useful when one wishes to
compare the results of this paper with those of \cite{[Sche]} and \cite{[Str]}.

\medskip

In Section \ref{Local} we shall encounter a Gauss sum, arising from a lattice
over the ring of integers in a local field, together with two coprime elements
$a$ and $c$ of that ring, with $c\neq0$. We shall now evaluate this Gauss sum in
the case where the ring is $\mathbb{Z}_{p}$. Let $M$ be an even $p$-adic lattice, and let $c\neq0$ and $a$ be elements of $\mathbb{Z}_{p}$ which are relatively prime. Let $\Delta_{M,c}$ be the cardinality of the kernel of multiplication by $c$ on $D_{M}$, and define a vector $\tilde{x}_{c} \in M^{*}$ as follows. If $p\neq2$ then $\tilde{x}_{c}=0$. For $p=2$ we choose a Jordan decomposition of $M$ and consider the lattice $M_{v_{2}(c)}$: If it comes with the index $II$ then again $\tilde{x}_{c}=0$. Otherwise, take an orthogonal
$\mathbb{Z}_{2}$-basis for it and let $\tilde{x}_{c}$ be the half the sum of
these basis vectors. Finally, define $a_{p}$ to be $a/p^{v_{p}(a)}$ if $a\neq0$.
Then we obtain
\begin{thm}
The Gauss sum $\sum_{\eta \in
M/cM}\chi_{p}\big(\frac{a}{c}\frac{\eta^{2}}{2}+a\frac{(\tilde{x}_{c},\eta)}{c}
\big)$ is well-defined and equals $p^{rk(M)v_{p}(c)/2}\sqrt{\Delta_{M,c}}\omega$
where $\omega=\prod_{pq|c}\gamma\big(q^{\kappa n}_{/t/II}(a_{p}c)\big)$ and
an empty product is defined (as always) to be 1. \label{M/cM}
\end{thm}

The index $/t/II$ appearing in $\omega$ means no index for odd $p$ and means $t$
or $II$ according to what appears in that component for $p=2$. We remark that if
$a=0$ then the product defining $\omega$ is empty, avoiding the ambiguity of
$a_{p}$.

\begin{proof}
First, replacing a summand $\eta$ by $\eta+c\lambda$ with $\lambda \in M$
changes the argument of $\chi_{p}$ by an element of $\mathbb{Z}_{p}$. Hence each
summand is indeed well-defined. It again suffices to verify the remaining
assertions for each Jordan component. If $M$ is $q^{\varepsilon n}_{/t/II}$ then
$\Delta_{M,c}$ equals $q^{n}$ if $v_{p}(q) \leq v_{p}(c)$ and equals
$p^{v_{p}(c)n}$ if $v_{p}(q) \geq v_{p}(c)$. It turns out to be more convenient
to multiply $\eta$ by $c_{p}$ (possible since $c_{p}\in\mathbb{Z}_{p}^{*}$), so
that the summand corresponding to $\eta$ is
$\chi_{p}\big(ac_{p}\frac{\eta^{2}}{2p^{v_{p}(c)}}+a\frac{(\eta,\tilde{x}_{c})}{
p^{v_{p}(c)}}\big)$. We also recall that $q|\eta^{2}$ for every $\eta$ in the
Jordan component.

We distinguish among three different cases. The first case is where $p=2$,
$v_{2}(q)=v_{2}(c)$, and we have an index $t$ (the case where
$\tilde{x}_{c}\neq0$). The second case occurs whenever $v_{p}(q) \geq v_{p}(c)$
but excluding the situation covered in the first case. The third case is where
$pq|c$. In the second case the argument of $\chi_{p}$ is in $\mathbb{Z}_{p}$ for
every $\eta$, implying the assertion since $\Delta_{M,c}=p^{v_{p}(c)n}$ and
$\omega=1$ (this deals with the case where $c\in\mathbb{Z}_{p}^{*}$, and in
particular where $a=0$).

To prove the third case, note that the terms corresponding to the indices $\eta$
and $\eta+\frac{p^{v_{p}(c)}}{q}\lambda$ with $\lambda \in M$ have the same
contribution. As this class coincides with $\eta$ if and only if $\lambda \in
qM$ (so that the difference is in $cM=p^{v_{p}(c)}M$---multiplication by the
$p$-adic unit $c_{p}$ does not change $M$ as a module), and $M/qM$ has $q^{n}$
elements, our sum equals $q^{n}$ times the sum of the same expression over
$M\big/\frac{p^{v_{p}(c)}}{q}M$. Moreover, we have $a_{p}=a$ since $p|c$
hence $a\in\mathbb{Z}_{p}^{*}$. Now, as $M=q^{\kappa n}_{/t/II}$ and
$v_{p}(c)>v_{p}(q)$, multiplying the bilinear form in $M$ by
$\frac{p^{v_{p}(c)}}{q^{2}}$ still gives a even lattice, which we denote $L$ and
whose symbol is $\big(\frac{p^{v_{p}(c)}}{q}\big)^{\kappa n}_{/t/II}$. We claim
that the sum in question, namely $\sum_{\eta \in
M/\frac{p^{v_{p}(c)}}{q}M}\chi_{p}\big(a_{p}c_{p}\frac{\eta^{2}}{2p^{v_{p}(c)}}
\big)$, equals $\sum_{\rho \in
L^{*}/L}\chi_{p}\big(a_{p}c_{p}\frac{\rho^{2}}{2}\big)$. Indeed,
$L^{*}$ is $\frac{q}{p^{v_{p}(c)}}L$ (which is $M$ with the bilinear form
divided by $p^{v_{p}(c)}$) and $L=\frac{p^{v_{p}(c)}}{q}L^{*}$, so the two sums
indeed coincide. Proposition \ref{gammap} now shows that the latter sum equals
$\frac{p^{nv_{p}(c)/2}}{\sqrt{q^{n}}}\gamma\big(L(a_{p}c_{p})\big)$ (as
$\Delta_{L(a_{p}c_{p})}=\Delta_{L}=\frac{p^{nv_{p}(c)}}{q^{n}}$). As the
original sum was $q^{n}$ times the latter, $\Delta_{M,c}=q^{n}$, and
$L(a_{p}c_{p})$ has the same Weil index as $L(q^{2}a_{p}c_{p})=M(a_{p}c)$, this
proves the asserted result for this case.

It remains to consider the first case, which can occur only if $c$ and $q$ are
even. Hence $a$ is odd. Similarly to the third case, we can take out a factor
of $2^{n(v_{2}(c)-1)}$ and carry out the summation on $M/2M$. Given an
orthogonal $\mathbb{Z}_{2}$-basis for $M$ and an element $\eta \in M/2M$, we
find that $\chi_{2}\big(ac_{2}\frac{\eta^{2}}{2q}\big)$ is $-1$ raised to the
power which is the sum of the coefficients in the presentation of $\eta$ using
this basis. On the other hand,
$\chi_{2}\big(a\frac{(\eta,\tilde{x}_{c})}{q}\big)$ is seen to yield the same
value. The product of these two elements thus equals 1 for every $\eta$ in
$M/2M$, the Gauss sum is $2^{nv_{2}(c)}$, and as in the second case this is the
value we need.

This proves the theorem.
\end{proof}

\section{Metaplectic Groups over Local Fields \label{Meta}}

The operators appearing in the local Weil representations which we seek to
evaluate are not the ones denoted $\mathbf{r}_{0}$ in \cite{[W]}, but rather
their multiples by appropriate roots of unity. These roots of unity can be given
explicitly in terms of Weil indices of the quadratic form on the lattice and of
related quadratic forms. In this Section we thus construct the metaplectic cover
of $SL_{2}(\mathbb{F})$ for a local field $\mathbb{F}\neq\mathbb{C}$ of
characteristic $\neq2$ as acting on $\mathbb{F}$-lattices in these terms. We
then proceed to review the splitting of the cover over the ring of integers in
odd residue characteristics, and relate the real and 2-adic double covers of
$SL_{2}(\mathbb{Z})$ explicitly. These results yield the required roots of unity
mentioned above.

\medskip

Let $\mathbb{F}$ be a local field of characteristic different from 2 which is
not $\mathbb{C}$. We denote $(a,b)_{\mathbb{F}}$ the Hilbert symbol of the two
elements $a$ and $b$ of $\mathbb{F}^{*}$. It is symmetric, bi-multiplicative,
with values in $\{\pm1\}$, and is essentially defined on pairs of elements of
$\mathbb{F}^{*}/(\mathbb{F}^{*})^{2}$. The paper \cite{[Ku1]} constructs
non-trivial finite covers of $SL_{2}(\mathbb{F})$ for any local field
$\mathbb{F}\neq\mathbb{C}$ using norm residue symbols. In particular, the double
cover $Mp_{2}(\mathbb{F})$ of $SL_{2}(\mathbb{F})$ is given in terms of the
Hilbert symbol. An element of $Mp_{2}(\mathbb{F})$ can be realized by a pair
$(A,\theta)$ with $A \in SL_{2}(\mathbb{F})$ and $\theta\in\{\pm1\}$, and the
product is defined by \[(A,\theta)(B,\psi)=(AB,\sigma(A,B)\theta\psi)\] where
$\sigma(A,B)$ is the cocycle denoted in \cite{[Ku1]} by $a(\sigma,\tau)$. The
formula is
\begin{equation}
\sigma(A,B)=(x(A),x(B))_{\mathbb{F}}(x(AB),-x(B)/x(A))_{\mathbb{F}},
\label{cocycle}
\end{equation}
where the $x$-image of a matrix $\binom{a\ \ b}{c\ \ d} \in SL_{2}(\mathbb{F})$
is $c$ if $c\neq0$ and $d$ if $c=0$.

The group denoted $Mp_{2}(\mathbb{R})$ in Section \ref{LatWRep} is isomorphic to
the double cover $Mp_{2}(\mathbb{R})$ defined in \cite{[Ku1]}. Indeed,
identifying $(A,\theta)$ from \cite{[Ku1]} with the element
$\big(A,\theta\sqrt{j(A,\tau)}\big)$ from Section \ref{LatWRep}, where the
argument of $\sqrt{j(A,\tau)}$ is in $\big[-\frac{\pi}{2},\frac{\pi}{2}\big)$
defines such an isomorphism---see Theorem 4.1 of \cite{[Str]} for a proof (this
theorem considers only $A \in SL_{2}(\mathbb{Z})$, but the proof extends to all
of $SL_{2}(\mathbb{R})$).

\medskip

Choose a non-trivial character $\lambda$ on $(\mathbb{F},+)$. Any other choice
is $\lambda$ composed with multiplication by an element of $\mathbb{F}^{*}$. Let
$V$ be an $\mathbb{F}$-lattice. The map $\psi:V \to V^{*}$ induced by the
bilinear form is a symmetric isomorphism, acting from the right as in
\cite{[W]}. Composition with $\lambda$ defines an isomorphism from the dual
vector space $V^{*}$ to the Pontryagin dual $\widehat{V}$, and $\lambda_{*}\psi$
is a symmetric isomorphism in the terminology of \cite{[W]}. Normalize the Haar
measure on $V$ accordingly. Since $ch\mathbb{F}\neq2$, the bilinear form on $V$
corresponds to a (unique) quadratic form $q:x\mapsto\frac{x^{2}}{2}$, and
$f=\lambda \circ q$ is a (quadratic) non-degenerate character of second degree
which is associated to $\lambda_{*}\psi$. Applying the process presented in
Section \ref{LatWRep} to $G=V$, we find that $Sp(V)$ contains a copy of
$SL_{2}(\mathbb{F})$ rather than just $SL_{2}(\mathbb{Z})$. The group of unitary
operators on $L^{2}(V)$ (or on the dense subspace $\mathcal{S}(V)$) which is
denoted $Mp(V)$ in Section 34 of \cite{[W]} is an $S^{1}$-cover of the subgroup
of $Sp(V)$ in which the entries from $End(V)$ (as a locally compact group) are
$\mathbb{F}$-linear. This group contains a double cover of this symplectic group
over $\mathbb{F}$, which is described in detail in Section 5 of \cite{[Ra]},
using symplectic notation. The (Weil) representation associated to $V$ sends
$Mp_{2}(\mathbb{F})$ to elements of this double cover which lie over
$SL_{2}(\mathbb{F})$ (see Theorem \ref{Vrep} below).

Let $A=\binom{a\ \ b}{c\ \ d} \in Sp_{\mathbb{F}}(V \times V)$, and assume that
$a$, $b$, $c$, and $d$ are $\mathbb{F}$-linear endomorphisms of $V$ (acting from
the right). \cite{[W]} provides formulae for the lift of such elements into
$Mp(V)$ in some cases, namely Eq. (16) there for invertible $c$ and the
appropriate combination $\mathbf{t}_{0}(\tilde{f})\mathbf{d}_{0}(\alpha)$ for
$c=0$. These formulae are (in our terminology and normalization) as follows: If
$c=0$ then
\begin{equation}
\mathbf{r}_{0}(A)\Phi(x)=\sqrt{|\det
a|_{\mathbb{F}}}\Phi(xa)\lambda\bigg[\frac{(xa,xb)}{2}\bigg], \label{Wc=0}
\end{equation}
while if $c$ is invertible then Eq. (16) of \cite{[W]} (with $f$ being the
quadratic character associated to $A$) states that
\[\mathbf{r}_{0}(A)\Phi(x)=\sqrt{|\det
c|_{\mathbb{F}}}\int_{V}\Phi(xa+uc)\lambda\bigg[\frac{(xa,xb)}{2}+(uc,xb)+\frac{
(uc,ud)}{2}\bigg]du.\] Here $|\cdot|_{\mathbb{F}}$ is the normalized absolute
value of $\mathbb{F}$. Using the symplectic relation $b=ac^{-1}d-(c^{*})^{-1}$,
a change of variables sends the latter equation to
\begin{equation}
\mathbf{r}_{0}(A)\Phi(x)=\frac{1}{\sqrt{|\det
c|_{\mathbb{F}}}}\int_{V}\Phi(y)\lambda\bigg[\frac{(yc^{-1}d,y)}{2}-(yc^{-1},
x)+\frac{(xac^{-1},x)}{2}\bigg]dy. \label{Wcinv}
\end{equation}
These formulae suffice for evaluating $\mathbf{r}_{0}(A)$ wherever $A \in
SL_{2}(\mathbb{F})$.

$Mp_{2}(\mathbb{F})$ acts on $L^{2}(V)$ or $\mathcal{S}(V)$ through a
representation, which we denote $\rho_{V/\mathbb{F}}$, in which
$\rho_{V/\mathbb{F}}(A,\theta)$ is the appropriate constant multiple of
$\mathbf{r}_{0}(A)$ from Eqs. \eqref{Wc=0} and \eqref{Wcinv} (so that the image
of $\rho_{V/\mathbb{F}}$ is contained in $Mp(V)$). We wish to have an intrinsic
expression for these coefficients, using the quadratic form on $V$ (rather than
the symplectic notation of \cite{[Ra]}). For this we denote, by a slight abuse
of notation, the character of second degree $\lambda\circ(tq)$ (where the
quadratic form $q$ is multiplied by some $t\in\mathbb{F}^{*}$) simply by $tf$.
Denoting $\dim V$ by $m$, we then have
\begin{thm}
$\rho_{V/\mathbb{F}}(A,\theta)$ equals
$\theta^{m}\overline{\gamma(cf)}\mathbf{r}_{0}(A)$ if $c\neq0$ and
$\theta^{m}\overline{\gamma(af)}\gamma(f)\mathbf{r}_{0}(A)$ if $c=0$. The
representation $\rho_{V/\mathbb{F}}$ is faithful for odd $m$ and factors through
a faithful representation of $SL_{2}(\mathbb{F})$ if $m$ is even. \label{Vrep}
\end{thm}

\begin{proof}
Assume first that $m=1$, and denote $u^{2}$ for some $0 \neq u \in V$ by $r$.
The result now follows from Definition 5.2, Theorem 5.3, and the remark after
Corollary 5.7 of \cite{[Ra]} by choosing $\chi(x)=\lambda(rx)$, since then
$\gamma(f)$ is $\gamma_{F}\big(\frac{1}{2}\chi\big)$ in the notation of
\cite{[Ra]} and $\gamma_{F}\big(t,\frac{1}{2}\chi\big)$ is
$\frac{\gamma(tf)}{\gamma(f)}$ for any $t\in\mathbb{F}^{*}$. For the general
case use an orthogonal basis for $V$ in order to decompose $V$ as the orthogonal
direct sum of 1-dimensional spaces and use Proposition \ref{prod}.
Alternatively, one has the equality $\gamma(f)\overline{\gamma(\alpha
f)}\overline{\gamma(\beta f)}\gamma(\alpha\beta
f)=(\alpha,\beta)_{\mathbb{F}}^{m}$ arising from the formula at the bottom of
page 176 of \cite{[W]}, using which one may verify that the product of two elements of the asserted image of $\rho_{V/\mathbb{F}}$ also lies in that image (because of Equation \eqref{cocycle}). The assertion concerning the parity of $m$ is trivial. This proves the theorem.
\end{proof}

It it important to mention that Theorem \ref{Vrep} is not a proper special case
of the results of \cite{[Ra]}. Indeed, restricting the formula of \cite{[Ra]} to
$SL_{2}(\mathbb{F})$ (with any choice of character $\chi$) does not always give
the coefficients presented in Theorem \ref{Vrep}. The reason for this is the
fact that in order to represent the bilinear form on $V$ we need to take a
different character $\chi$ for every basis element, while \cite{[Ra]} considers
only one such character.

\medskip

Let us assume that $\mathbb{F}$ is non-archimedean (of characteristic $\neq2$),
with ring of integers $\mathcal{O}$, uniformizer $\pi$ (so that the unique
maximal ideal in $\mathcal{O}$ is $\pi\mathcal{O}$), and valuation $v$. It is
shown in \cite{[Ku2]} and \cite{[Ge]} that the metaplectic double cover splits
over the group $\Gamma_{1}(4,\mathcal{O})$ consisting of those matrice
$\binom{a\ \ b}{c\ \ d} \in SL_{2}(\mathcal{O})$ in which $4|c$ and $a \equiv
d\equiv1(\mathrm{mod\ }4)$. Explicitly, let
\[\iota:\Gamma_{1}(4,\mathcal{O}) \to
Mp_{2}(\mathbb{F}),\quad\iota(A)=\left\{\begin{array}{ll} (A,1) & c=0 \\
\big(A,(a,\pi^{v(c)})_{\mathbb{F}}\big)=\big(A,(d,\pi^{v(c)})_{\mathbb{F}}\big)
& c\neq0.\end{array}\right.\] Then we have
\begin{thm}
The map $\iota$ is a group-theoretic lift of $\Gamma_{1}(4,\mathcal{O})$ into
$Mp_{2}(\mathbb{F})$. \label{Gamma14O}
\end{thm}

This result is stated in Proposition 2.8 of \cite{[Ge]} and proven as Theorem 2
of \cite{[Ku2]}: The intersection of the group denoted $K^{N}$ with $N=4$ in
\cite{[Ge]} with $SL_{2}(\mathcal{O})$ gives precisely
$\Gamma_{1}(4,\mathcal{O})$. In these references the characteristic of
$\mathbb{F}$ is assumed to be 0, but the proof holds (at least for the double
cover) for any characteristic $\neq2$. In fact, the sign attached to $A \in
SL_{2}(\mathbb{F})$ in these references is $(d,c)_{\mathbb{F}}$ if $cd\neq0$ and
$\pi|c$ and 1 otherwise. However, using Proposition 2.1 of \cite{[Ge]} one can
show that this is equivalent to $(d,\pi^{v(c)})_{\mathbb{F}}$ for $c\neq0$, and
the equality $(a,\pi^{v(c)})_{\mathbb{F}}=(d,\pi^{v(c)})_{\mathbb{F}}$ follows
from similar considerations since $ad=1+bc$. If $a$ or $d$ vanish then
$\pi^{v(c)}=1$, and we define $(0,1)_{\mathbb{F}}$ to be 1. The observation that
for $\mathbb{F}=\mathbb{Q}_{p}$ with odd $p$ the Hilbert symbol
$(a,p^{v_{p}(c)})_{\mathbb{F}}=(d,p^{v_{p}(c)})_{\mathbb{F}}$ coincides with
the (extended) Legendre symbol
$\big(\frac{a}{p^{v_{p}(c)}}\big)=\big(\frac{d}{p^{v_{p}(c)}}\big)$ (also if
$ad=0$) will turn out useful in the sequel.

For odd residue characteristic we have
$\Gamma_{1}(4,\mathcal{O})=SL_{2}(\mathcal{O})$, so that the composition
$\rho_{V/\mathbb{F}}\circ\iota$ is a (faithful) representation of
$SL_{2}(\mathcal{O})$. This is why the representation of $Mp_{2}(\mathbb{Z})$
obtained from the process of Section \ref{LatWRep} factors through
$SL_{2}(\mathbb{Z})$ in this case. On the other hand, if $\mathbb{F}$ is an
extension of $\mathbb{Q}_{2}$ then $\Gamma_{1}(4,\mathcal{O})$ is a proper
subgroup of $SL_{2}(\mathcal{O})$, and in general Theorem \ref{Gamma14O} cannot
be extended to the full group $SL_{2}(\mathcal{O})$. This follows, for example,
from Theorem \ref{Mp2ZRQ2} below for $\mathbb{F}=\mathbb{Q}_{2}$, and Theorem
\ref{Vrep} implies that the same occurs wherever the degree of the extension
$\mathbb{F}/\mathbb{Q}_{2}$ is odd. It is likely that this occurs for every
extension of $\mathbb{Q}_{2}$ (see also the discussion at the end of Section
\ref{OddGen}).

\medskip

The two ``metaplectic groups over $\mathbb{Z}$'', namely the one embedded in
$Mp_{2}(\mathbb{R})$ as in Section \ref{LatWRep} and the one embedded in
$Mp_{2}(\mathbb{Q}_{2})$, must be isomorphic. This is so, since the double
cover of the ad\`{e}lic metaplectic group splits (hence the product of all the
local cocycles is trivial), and we have seen in Theorem \ref{Gamma14O} that the
metaplectic cover splits over every $\mathbb{Z}_{p}$ for every odd $p$. For our
purposes we need the explicit isomorphism, since the global representation
$\rho_{M}$ of Section \ref{LatWRep} is defined on $Mp_{2}(\mathbb{Z}) \leq
Mp_{2}(\mathbb{R})$ while the 2-adic representation $\rho_{M_{2}}$ is given in
terms of $Mp_{2}(\mathbb{Z}) \leq Mp_{2}(\mathbb{Z}_{2})$. We use the
``abstract'' notation $(A,\theta)$ for elements of $Mp_{2}(\mathbb{Q}_{2})$ and
the ``modular'' notation $\big(A,\theta\sqrt{j(A,\tau)}\big)$, with
$\sqrt{j(A,\tau)}$ having argument in $\big[-\frac{\pi}{2},\frac{\pi}{2}\big)$,
for $Mp_{2}(\mathbb{R})$. We define a map $i$ from the ''modular''
$Mp_{2}(\mathbb{Z}) \leq Mp_{2}(\mathbb{R})$ to $Mp_{2}(\mathbb{Q}_{2})$ by
\[i\big(A,\theta\sqrt{j(A,\tau)}\big)=\left\{\begin{array}{ll} (A,\theta) & c=0
\\
\big(A,\big(\frac{a}{c_{2}}\big)\theta\big)=\big(A,\big(\frac{d}{c_{2}}
\big)\theta\big) & c\neq0.\end{array}\right.\] We now prove
\begin{thm}
The map $i$ is a group injection. \label{Mp2ZRQ2}
\end{thm}

\begin{proof}
The proof reduces to expressing the local-to-global properties explicitly, and
comparing with the maps $\iota$ and $i$. Given two matrices $A$ and $B$ in
$SL_{2}(\mathbb{Q})$, the product $\prod_{p\leq\infty}\sigma_{p}(A,B)$ is finite
and equals unity by the Hilbert reciprocity law (this is equivalent to the
splitting of the ad\`{e}lic metaplectic group over $SL_{2}(\mathbb{Q})$). Here
$\sigma_{p}(A,B)$ is the value of the cocycle on $A$ and $B$ considered as
matrices in $SL_{2}(\mathbb{Q}_{p})$, with $\mathbb{Q}_{\infty}=\mathbb{R}$. The
product of $\big(A,\theta\sqrt{j(A,\tau)}\big)$ and
$\big(B,\psi\sqrt{j(B,\tau)}\big)$ for $A$ and $B$ in $SL_{2}(\mathbb{Q}) \leq
SL_{2}(\mathbb{R})$ thus yields the element
$\big(AB,\theta\psi\prod_{p<\infty}\sigma_{p}(A,B)\sqrt{j(AB,\tau)}\big)$.

For $A \in SL_{2}(\mathbb{Z})$ and odd $p$, write $\iota(A) \in
Mp_{2}(\mathbb{Q}_{p})$ from Theorem \ref{Gamma14O} as $(A,\delta_{A,p})$. The
equality $\sigma_{p}(A,B)=\delta_{AB,p}\delta_{A,p}\delta_{B,p}$ thus holds for
every such $A$, $B$, and $p$ by that theorem. Furthermore, write
$i\big(A,\theta\sqrt{j(A,\tau)}\big)$ as $(A,\theta\eta_{A})$. As for $c\neq0$
the coefficient $\delta_{A,p}$ equals
$\big(\frac{a}{p^{v_{p}(c)}}\big)=\big(\frac{d}{p^{v_{p}(c)}}\big)$, we deduce
that $\eta_{A}=\prod_{2<p<\infty}\delta_{A,p}$ (recall that in our convention,
the Legendre symbols over $|c_{2}|$ and over $c_{2}$ coincide). It follows that
\[\big(A,\theta\sqrt{j(A,\tau)}\big)\big(B,\psi\sqrt{j(B,\tau)}\big)=\big(AB,
\theta\psi\eta_{AB}\eta_{A}\eta_{B}\sigma_{2}(A,B)\sqrt{j(AB,\tau)}\big).\] But
applying $i$ to the right hand side gives the product, in
$Mp_{2}(\mathbb{Q}_{2})$, of the $i$-images of the elements appearing on the
left hand side. This completes the proof of the theorem since $i$ is clearly
injective.
\end{proof}

Comparing the actions of $T$ and $S$ shows that the process from Section
\ref{LatWRep} yields the restriction of $\rho_{V/\mathbb{F}}\circ\iota$ to
$SL_{2}(\mathbb{Z})$ for lattices over a local field $\mathbb{F}$ with odd
residue characteristic and the representation $\rho_{V/\mathbb{Q}_{2}} \circ i$
of $Mp_{2}(\mathbb{Z})$ for 2-adic vector spaces. Moreover, the congruence $a
\equiv d\equiv1(\mathrm{mod\ }4)$ for elements of the group
$\Gamma_{1}(4,\mathcal{O})$ considered in Theorem \ref{Gamma14O} implies that
when taking $\mathbb{F}=\mathbb{Q}_{2}$ (and $c\neq0$) in that theorem, the
Hilbert symbols can be replaced by $\big(\frac{2}{a}\big)^{v_{2}(c)}$ or
$\big(\frac{2}{d}\big)^{v_{2}(c)}$. Using the quadratic reciprocity law and
taking care of the convention difference for Legendre symbols, one verifies that
$i^{-1}\circ\iota$ coincides with the section denoted $s$ in \cite{[BS]} and
with the one appearing in Lemma 5.3 of \cite{[B3]} for $\Gamma_{1}(4)$.

\smallskip

At this point we remark about the connection to theta functions. The tensor
product $\rho_{M}=\bigotimes_{p}\rho_{M_{p}}$ may be seen, in the point of view
of Section \ref{Decom}, as a representation of (a double cover of) the group
$SL_{2}(\widehat{\mathbb{Z}})$ on a finite-dimensional subspace of
$M_{\mathbb{A}_{f}}$. Here $\mathbb{A}_{f}$ is the ring of finite ad\`{e}les of
$\mathbb{Q}$ and $\widehat{\mathbb{Z}}$ is the compact open subring
$\prod_{p}\mathbb{Z}_{p}$. The splitting of the double cover
$Mp_{2}(\mathbb{A})$ (based on the full ad\`{e}le ring) over
$SL_{2}(\mathbb{Q})$, which is mentioned in the proof of Theorem \ref{Mp2ZRQ2},
means that in some sense, the representation arising from
$Mp_{2}(\mathbb{A}_{f})$ behaves in the same way (up to dualization,
normalization, and other conventions) as the representation on Schwartz
functions in the infinite place. The natural Schwartz functions on
$M_{\mathbb{R}}$ which are usually considered in this context are the theta
functions of cosets of $M$ inside $M^{*}$, a vector-valued version of which
being modular with representation $\rho_{M}$ by Theorem 4.1 of \cite{[B1]}. To
emphasize this connection, we remark that Theorem 4.1 of \cite{[B1]} is, in
fact, a special case of Theorem 4 of \cite{[W]}, with $G=M_{\mathbb{R}}$ and
$\Gamma=M$. Theorem 6.1 of \cite{[Ge]} also shows a connection between
representations on spaces of theta functions (as the discrete non-cuspidal
spectrum of the associated Ad\'{e}lic space) and the corresponding Ad\'{e}lic
Weil representation. On the other hand, for our purposes of determining
$\rho_{M}(A)$ explicitly for each $A \in Mp_{2}(\mathbb{Z})$, one does not need
to use this connection.

\section{Evaluation of Local Operators \label{Local}}

In this Section we evaluate, for a non-archimedean $\mathbb{F}$, the operators
$\mathbf{r}_{0}(A)$ for a matrix $A \in SL_{2}(\mathbb{F})$ with integral
entries, on certain Schwartz functions on vector spaces over $\mathbb{F}$.
Applying this to $\mathbb{F}=\mathbb{Q}_{p}$ and multiplying by the roots of
unity from Section \ref{Meta} then combines, in the next Section, to yield the
main result of this paper.

\medskip

Let $\mathbb{F}$ be a non-archimedean local field of characteristic $\neq2$,
with normalized valuation $v$, ring $\mathcal{O}$ of integers, uniformizer
$\pi$, and residue field of cardinality $q$. Choose a character $\lambda$ on
$(\mathbb{F},+)$ such that $\lambda(x\mathcal{O})=1$ if and only if
$x\in\mathcal{O}$. Let $M$ be an even $\mathcal{O}$-lattice of rank $m$.
Consider, for some $\gamma \in D_{M}$, the characteristic function
$E_{M+\gamma}$ of a coset $M+\gamma \subseteq M_{\mathbb{F}}$. This function is
in $\mathcal{S}(M_{\mathbb{F}})$. We wish to evaluate
$\mathbf{r}_{0}(A)E_{M+\gamma}$ for $A=\binom{a\ \ b}{c\ \ d} \in
SL_{2}(\mathcal{O})$. Note that if $c\neq0$ then $M/cM$ is finite, hence no
convergence issues arise in the sums appearing below. The first step is
described in the following
\begin{prop}
If $c=0$ then $\mathbf{r}_{0}(A)E_{M+\gamma}$ equals
$\lambda\big(bd\frac{\gamma^{2}}{2}\big)E_{M+d\gamma}$. If $c\neq0$
then it is \[\frac{1}{q^{mv(c)/2}\sqrt{\Delta_{M}}}\sum_{\delta \in
D_{M}}\bigg[\sum_{\eta \in
M/cM}\lambda\bigg(\frac{a}{c}\frac{(\delta+\eta-d\gamma)^{2}}{2}\bigg)\bigg]
\lambda\bigg(b(\gamma,\delta)-bd\frac{\gamma^{2}}{2}\bigg)E_{M+\delta}.\]
\label{r0A}
\end{prop}

\begin{proof}
If $c=0$ then Eq. \eqref{Wc=0} gives
$\mathbf{r}_{0}(A)E_{M+\gamma}(x)=\lambda\bigg(ab\frac{x^{2}}{2}\bigg)E_{
M+\gamma}(ax)$ for $x \in M_{\mathbb{F}}$, since $\det a=a^{m}$ has valuation
0. As $a=\frac{1}{d}\in\mathcal{O}^{*}$ the characteristic function becomes
$E_{M+d\gamma}(x)$, and for $x \in M+d\gamma$ the argument of $\lambda$ becomes
$ab\frac{d^{2}\gamma^{2}}{2}=bd\frac{\gamma^{2}}{2}$ up to
$\mathcal{O}=\ker\lambda$. This covers the case $c=0$.

If $c\neq0$ then the coefficient $\frac{1}{\sqrt{|\det c|_{\mathbb{F}}}}$ in Eq.
\eqref{Wcinv} equals $q^{mv(c)/2}$, and we have to evaluate the integral.
Decompose $M+\gamma$ as $cM+\gamma+\eta$ for $\eta \in M/cM$, and then
substituting $y=\gamma+\eta+v$ for each $\eta$ yields
\[\int_{M+\gamma}\lambda\bigg(\frac{d}{c}\frac{y^{2}}{2}-\frac{(y,x)}{c}+\frac{a
}{c}\frac{x^{2}}{2}\bigg)dy=\] \[=\sum_{\eta \in
M/cM}\int_{cM}\lambda\bigg(\frac{d}{c}\frac{(\gamma+\eta)^{2}}{2}+\frac{d}{c}
(\gamma+\eta,v)+\frac{d}{c}\frac{v^{2}}{2}-\frac{(\gamma+\eta,x)}{c}-\frac{(v,x)
}{c}+\frac{a}{c}\frac{x^{2}}{2}\bigg)dv.\] With $v=cu$, $u \in M$ the integral
corresponding to $\eta$ becomes
\[q^{-mv(c)}\int_{M}\lambda\bigg(\frac{d}{c}\frac{(\gamma+\eta)^{2}}{2}
+\overbrace{d(\gamma+\eta,u)}+\overbrace{cd\frac{u^{2}}{2}}-\frac{(\gamma+\eta,
x)}{c}-(u,x)+\frac{a}{c}\frac{x^{2}}{2}\bigg)du.\] Both over-braced elements are
in $\mathcal{O}$, and the expression $\int_{M}\lambda\big(-(x,u)\big)du$
vanishes for $x \not\in M^{*}$ and gives the normalized measure
$\frac{1}{\sqrt{\Delta_{M}}}$ of $M$ if $x \in M^{*}$. We thus consider only $x
\in M^{*}$, so that $x \in M+\delta$ for some $\delta \in D_{M}$. Write
$x=\delta+w$ with $w \in M$, and $\mathbf{r}_{0}(A)E_{M+\gamma}(\delta+w)$
becomes \[\frac{1}{q^{mv(c)/2}\sqrt{\Delta_{M}}}\sum_{\eta \in
M/cM}\lambda\bigg(\frac{d}{c}\frac{(\gamma+\eta)^{2}}{2}-\frac{(\gamma+\eta,
\delta+w)}{c}+\frac{a}{c}\frac{(\delta+w)^{2}}{2}\bigg).\] Now, $w \in M$, and
multiplication by $a$ is injective on $M/cM$ since $a$ and $c$ are coprime.
Hence replacing $\eta$ by $aw-a\eta$ is an admissible change of variable on
$M/cM$. Expand all parentheses, and write $\frac{ad}{c}=\frac{1}{c}+b$ and
$\frac{a^{2}d}{c}=\frac{a}{c}+ab$. As the terms $b(\gamma,w)$,
$ab\frac{w^{2}}{2}$,  $ab(\eta,w)$, $b(\gamma,\eta)$, and $ab\frac{\eta^{2}}{2}$
are all in  $\mathcal{O}$, this transforms
$\mathbf{r}_{0}(A)E_{M+\gamma}(\delta+w)$ to the form
\begin{equation}
\frac{1}{q^{mv(c)/2}\sqrt{\Delta_{M}}}\sum_{\eta \in
M/cM}\lambda\bigg(\frac{d}{c}\frac{\gamma^{2}}{2}-\frac{(\gamma,\delta+\eta)}{
c} +\frac{a}{c}\frac{(\delta+\eta)^{2}}{2}\bigg). \label{r0AGauss}
\end{equation}
As the expression in Eq. \eqref{r0AGauss} is independent of $w$, we deduce that
$\mathbf{r}_{0}(A)E_{M+\gamma}$ is a linear combination of
$\{E_{M+\delta}\}_{\delta \in D_{M}}$. Expanding $\frac{1}{c}=\frac{ad}{c}-b$
and $\frac{d}{c}=\frac{ad^{2}}{c}-bd$ turns the Gauss sum in Eq.
\eqref{r0AGauss} to \[\sum_{\eta \in
M/cM}\lambda\bigg(\frac{a}{c}\frac{(\delta+\eta-d\gamma)^{2}}{2}
+b(\gamma,\delta+\eta)-bd\frac{\gamma^{2}}{2}\bigg),\] which completes the proof
of the proposition as $b(\gamma,\eta)\in\mathcal{O}$.
\end{proof}

Note that the Gauss sum in Proposition \ref{r0A} is well-defined, i.e.,
independent of the elements from $M^{*}$ representing $\gamma$ and $\delta$
in $D_{M}=M^{*}/M$ and of the element of $M$ which represents $\eta \in M/cM$.

We now identify $V_{\rho_{M}}=\mathbb{C}[M^{*}/M]$ with the subspace
$\bigoplus_{\gamma \in D_{M}}\mathbb{C}E_{M+\gamma}$ of
$\mathcal{S}(M_{\mathbb{F}})$ via $e_{\gamma} \leftrightarrow E_{M+\gamma}$ as
in Section \ref{Decom}, and deduce the following
\begin{cor}
If $c=0$ then
$\mathbf{r}_{0}(A)e_{\gamma}$ equals
$\lambda\big(bd\frac{\gamma^{2}}{2}\big)e_{d\gamma} $, while for $c\neq0$ it
equals \[\frac{1}{q^{mv(c)/2}\sqrt{\Delta_{M}}}\sum_{\beta \in
D_{M}}\!\bigg[\sum_{\eta \in
M/cM}\!\!\lambda\bigg(\!\frac{a}{c}\frac{\eta^{2}}{2}+a\frac{(\beta,\eta)}{c}
\bigg)\bigg]\lambda\bigg(\!\frac{a}{c}\frac{\beta^{2}}{2}+b(\gamma,
\beta)+bd\frac{\gamma^{2}}{2}\bigg)e_{\beta+d\gamma}.\] \label{r0ACDM}
\end{cor}

\begin{proof}
This is just Proposition \ref{r0A} in the $V_{\rho_{M}}$ terminology, after
substituting $\delta=\beta+d\gamma$ in the formula for $c\neq0$.
\end{proof}

The advantage of Corollary \ref{r0ACDM} over Proposition \ref{r0A} is that the
Gauss sum is now independent of $\gamma$. We remark again that taking
$\mathbb{F}=\mathbb{Q}_{p}$ and $\lambda=\chi_{p}$ in Corollary \ref{r0ACDM}
yields the action of $T_{f}$ for $A=T$ and of $\widetilde{S}_{f}$ for $A=S$,
from which Observation \ref{Zpext} follows by the uniqueness of the scalar
required for obtaining $S_{f}$ from $\widetilde{S}_{f}$ in Section
\ref{LatWRep}.

\medskip

Let $M$ be an even lattice over an integral domain $R$ whose fraction field
$\mathbb{K}$ with $ch\mathbb{K}\neq2$, and let $c \in R$. Multiplication by $c$
yields the exact sequence $0 \to D_{M,c} \to D_{M} \to D_{M}^{c}\to0$ (see
Section 2 of \cite{[Sche]}), and $D_{M,c}$ and $D_{M}^{c}$ are orthogonal
complements in the non-degenerate $\mathbb{K}/R$-valued pairing on $D_{M}$.
Moreover, the map taking $\mu \in D_{M,c}$ to
$c\frac{\mu^{2}}{2}+(\beta,\mu)\in\mathbb{K}/R$ is linear on $D_{M,c}$ for any
$\beta \in D_{M}$, and we denote by $D_{M}^{c*}$ the set of those $\beta \in
D_{M}$ for which this map is identically $0\in\mathbb{K}/R$. This set is a coset
of $D_{M}^{c}$ inside $D_{M}$---see Proposition 2.1 of \cite{[Sche]} for the
case $R=\mathbb{Z}$, and the proof holds equally well for the more general
setting. We choose an element $x_{c}$ in the coset (in future applications we
shall specify the choice), so that any $\beta \in D_{M}^{c*}$ is $x_{c}+c\alpha$
for some $\alpha \in D_{M}$ which is well-defined up to $D_{M,c}$. Proposition
2.2 of \cite{[Sche]} (which generalizes as well) shows that the element
$\frac{\beta_{c}^{2}}{2}=c\frac{\alpha^{2}}{2}+(x_{c},\alpha)$ of $\mathbb{K}/R$
is well-defined (i.e., independent of the choice of $\alpha$). Note that this
element depends on the choice of $x_{c}$, but we consider $x_{c}$ as a pre-fixed
element of $D_{M}$. If $2 \in R^{*}$ then $D_{M}^{c*}=D_{M}^{c}$, so that the
natural choice in this case is to take $x_{c}=0$ for every $c$. These sets
$D_{M,c}$, $D_{M}^{c}$, and $D_{M}^{c*}$ are additive with respect to
orthogonal direct sums, and in case $R$ is the ring of integers $\mathcal{O}$ in
a global field of characteristic $\neq2$, belonging to any of them is a local
property with respect to a decomposition as in Section \ref{Decom}. In the
latter case we denote the cardinalities of $D_{M,c}$ by $\Delta_{M,c}$ (as in
Theorem \ref{M/cM} above). All these observations hold also for $c=0$, where
$D_{M,0}=D_{M}$, $D_{M}^{0}=D_{M}^{0*}=\{0\}$, $x_{0}=0$, and
$\frac{\beta_{0}}{2}=0\in\mathbb{K}/R$.

Returning to the case of local, non-archimedean $\mathbb{F}$, we now prove
\begin{lem}
The Gauss sum in Corollary \ref{r0ACDM} vanishes for $\beta \not\in D_{M}^{c*}$.
\label{DMc*}
\end{lem}

\begin{proof}
The assertion is vacuous if $c\in\mathcal{O}^{*}$, hence assume otherwise. It
follows that $a\in\mathcal{O}^{*}$. Take $\rho \in M^{*}$ such that $\rho+M \in
D_{M,c}$, and change the summation index $\eta$ to $\eta+c\rho$. This multiplies
the Gauss sum by $\lambda\big(ac\frac{\rho^{2}}{2}+a(\beta,\rho)\big)$. If
$\beta \not\in D_{M}^{c*}$ then this multiplier differs from 1 for some $\rho$
by our assumption on $a$ and $\ker\lambda$. Hence the sum must vanish. This
proves the lemma.
\end{proof}

We now choose $\tilde{x}_{c} \in M^{*}$ such that $x_{c}=\tilde{x}_{c}+M \in
D_{M}^{c*}$, and obtain
\begin{cor}
If $c\neq0$ then $\mathbf{r}_{0}(A)e_{\gamma}$ is
\[\frac{\lambda\big(\frac{a}{c}\frac{\tilde{x}_{c}^{2}}{2}\big)}{q^{mv(c)/2}
\sqrt{\Delta_{M}}}\sum_{\eta \in
M/cM}\lambda\bigg(\frac{a}{c}\frac{\eta^{2}}{2}+a\frac{(\tilde{x}_{c},\eta)}{c}
\bigg)\sum_{\beta \in
D_{M}^{c*}}\lambda\bigg(a\frac{\beta_{c}^{2}}{2}+b(\gamma,\beta)+bd\frac{\gamma^
{2}}{2}\bigg)e_{\beta+d\gamma}.\] \label{finloc}
\end{cor}

\begin{proof}
Take only $\beta \in D_{M}^{c*}$ in the sum from Corollary \ref{r0ACDM}, write
such $\beta$ as $x_{c}+c\alpha$ with $\alpha \in M^{*}$, and use the definition
of $\frac{\beta_{c}^{2}}{2}$.
\end{proof}

Note that even though $\frac{\beta_{c}^{2}}{2}$ depends only on $x_{c} \in
D_{M}$, the term $\lambda\big(\frac{a}{c}\frac{\tilde{x}_{c}^{2}}{2}\big)$ and
the Gauss sum depend on the particular element $\tilde{x}_{c} \in M^{*}$ because
of the division by $c$. Their product depends only on $x_{c}$, but in the
following we evaluate each term separately (with a specific choice of
$\tilde{x}_{c}$).

Consider now the case $R=\mathbb{Z}_{2}$. Section 2 of \cite{[Sche]} asserts
that the element $\tilde{x}_{c}$ used in Theorem \ref{M/cM} satisfies $x_{c} \in
D_{M}^{c*}$. Take $\lambda=\chi_{2}$ on $\mathbb{F}=\mathbb{Q}_{2}$, and recall
that if $c$ is odd then the Jordan component with $q=2^{v_{2}(c)}=1$ comes with
the index $II$ (for $M$ to be even), while if $c$ is even then $a=a_{2}$. A
straightforward calculation now evaluates the external coefficient in Corollary
\ref{finloc} to be
\begin{lem}
If the chosen Jordan decomposition of $M$ involves the component
$(2^{v_{2}(c)})^{\kappa n}_{t/II}$ then
$\chi_{2}\big(\frac{a}{c}\frac{\tilde{x}_{c}^{2}}{2}\big)$ equals
$\zeta_{8}^{a_{2}c_{2}t}$, where for an index $II$ we take $t=0$. \label{xcnorm}
\end{lem}
If $M$ is a $\mathbb{Z}$-lattice then we define $x_{c} \in D_{M} $ to be the
image of the thus defined $x_{c} \in D_{M_{2}}$ inside $D_{M}$ (in
correspondence with $x_{c}=0 \in D_{M_{p}}$ for odd $p$).

\section{General Formulae for $\rho_{M}$ \label{Main}}

In this Section we derive the main result of this paper, i.e., the action of the
general element of $Mp_{2}(\mathbb{Z})$ via the representation $\rho_{M}$.

\medskip

Let $M$ be an even lattice, and take an element
$\big(A,\theta\sqrt{j(A,\tau)}\big)$ of $Mp_{2}(\mathbb{Z})$. We can now
evaluate the image of this element under $\rho_{M}=\bigotimes_{p}\rho_{M_{p}}$
by specializing the results of Sections \ref{Meta} and \ref{Local} to
$\mathbb{F}=\mathbb{Q}_{p}$ and $\lambda=\chi_{p}$. Before we get to the final
formulae, we remark that the formula from part (i) of Proposition 1.6 of
\cite{[Sn]} is obtained as the tensor product of the $p$-adic
$\mathbf{r}_{0}(A)$. Indeed, in case $c=0$ we have just $a=d=\pm1$ and the
assertion follows from either Proposition \ref{r0A} or Corollary \ref{r0ACDM},
and if $c\neq0$ the one presents $\mathbf{r}_{0}(A)$ as in Eq. \eqref{r0AGauss}
and establishes the formula in question. As this formula is stated for the
actual $\rho_{M}$-image (rather than just the tensor product of the operators
$\mathbf{r}_{0}(A)$), it has to be multiplied by the coefficients from Theorem
\ref{Vrep} for all $p$, which yields $\zeta_{8}^{-sgn(c)sgn(M)}$ if $c\neq0$
(here $sgn(c)$ is just the usual sign $\frac{c}{|c|}\in\{\pm1\}$ of $c$) and
$\zeta_{8}^{(1-a)sgn(M)}$ if $c=0$ (and $a=d=\pm1$).

\smallskip

By putting $m=rk(M)$, we can now state and prove the main result.
\begin{thm}
For any element $\big(A,\theta\sqrt{j(A,\tau)}\big) \in Mp_{2}(\mathbb{Z})$
(with $\sqrt{j(A,\tau)}$ having its argument in
$\big[-\frac{\pi}{2},\frac{\pi}{2}\big)$ as usual) we have that
$\rho_{M}\big(A,\theta\sqrt{j(A,\tau)}\big)e_{\gamma}$ is
\[\prod_{p}\xi_{p}\cdot\frac{\sqrt{\Delta_{M,c}}}{\sqrt{\Delta_{M}}} \sum_{\beta
\in D_{M}^{c*}}\mathbf{e}\bigg(a\frac{\beta_{c}^{2}}{2}+b(\gamma,\beta)+bd\frac{
\gamma^{2}}{2}\bigg)e_{\beta+d\gamma},\] where $D_{M}^{c*}$, $\Delta_{M,c}$, and
$\frac{\beta_{c}^{2}}{2}$ are defined in the paragraph preceding Lemma
\ref{DMc*}. The root of unity $\xi_{p}$ defined to be
$\big(\frac{a_{p}}{\Delta_{M_{p}}}\big)\prod_{q\not|c}\overline{\gamma\big(q^{
\kappa n}(a_{p}c)\big)}$ for odd $p$ and
\[\theta^{m}\bigg(\frac{a}{c_{2}}\bigg)^{m}(-1)^{m\varepsilon(a_{2}
)\varepsilon(c_{2})}\bigg(\frac{2^{v_{2}(c)}}{a_{2}}\bigg)^{m}\bigg(\frac{
\Delta_{M_{2}}}{a_{2}}\bigg)\gamma(M_{2})^{a_{2}-1}\prod_{q\not|c}\overline{
\gamma\big(q^{\kappa n}_{t/II}(a_{2}c)\big)}\] for $p=2$, in case $ac\neq0$. For
$a=0$ we have just $\xi_{p}=\overline{\gamma\big(M_{p}(c)\big)}$ for odd $p$ and
$\xi_{2}=\theta^{m}\overline{\gamma\big(M_{2}(c)\big)}$, while if $c=0$ then
$\xi_{p}$ equals $\big(\frac{a_{p}}{\Delta_{M_{p}}}\big)$ for odd $p$ and
$\xi_{2}=\theta^{m}\gamma(M_{2})^{1-a}$. \label{final}
\end{thm}

\begin{proof}
Write $\mathbf{r}_{0}(A)e_{\gamma_{p}}$ as in Corollary \ref{finloc}, apply
Lemma \ref{xcnorm}, and evaluate the Gauss sum by Theorem \ref{M/cM}. This
yields
\[\omega_{p}\frac{\sqrt{\Delta_{M_{p},c}}}{\sqrt{\Delta_{M_{p}}}}\sum_{\beta_{p}
\in
D_{M_{p}}^{c*}}\chi_{p}\bigg(a\frac{\beta_{p,c}^{2}}{2}+b(\gamma_{p},\beta_{p}
)+bd\frac{\gamma_{p}^{2}}{2}\bigg)e_{\beta_{p}+d\gamma_{p}},\] where
$\omega_{p}$ is the factor $\prod_{pq|c}\gamma\big(q^{\kappa
n}_{/t/II}(a_{p}c)\big)$ appearing in Theorem \ref{M/cM}, times the factor from
Lemma \ref{xcnorm} for $p=2$. We may take the tensor product over all primes
$p$, since for $p$ not dividing $Nc$ (or $\Delta_{M}c$), $\rho_{M_{p}}$ is
1-dimensional and all the coefficients are 1. As
$\Delta_{M}=\prod_{p}\Delta_{M_{p}}$, $\Delta_{M,c}=\prod_{p}\Delta_{M_{p},c}$,
and $\beta \in D_{M}$ lies in $D_{M}^{c*}$ if and only if $\beta_{p} \in
D_{M_{p}}^{c*}$ for every $p$, this tensor product yields the asserted summation
and real constant. It thus remains to verify that multiplying $\omega_{p}$ by
the roots of unity distinguishing $\mathbf{r}_{0}(A)$ from $\rho_{M_{p}}(A)$
yields the asserted $\xi_{p}$.

In Section \ref{Meta} we evaluated $\rho_{M_{p}}$ as
$\rho_{M_{\mathbb{Q}_{p}}/\mathbb{Q}_{p}}\circ\iota$ for odd $p$ and
$\rho_{M_{\mathbb{Q}_{2}}/\mathbb{Q}_{2}} \circ i$ for $p=2$. The root of unity
appearing in $\rho_{M_{p}}\big(A,\theta\sqrt{j(A,\tau)}\big)$ is therefore
$\omega_{p}\big(\frac{a_{p}}{p^{v_{p}(c)}}\big)^{m}\overline{\gamma\big(M_{p}
(c)\big)}$ for odd $p$ and
$\omega_{2}\theta^{m}\big(\frac{a}{c_{2}}\big)^{m}\overline{\gamma\big(M_{2}
(c)\big)}$ for $p=2$. This proves the assertion for the case $a=0$, since then
$c=\pm1$ and therefore $\omega_{p}=1$ for all $p$,
$\big(\frac{a_{p}}{p^{v_{p}(c)}}\big)=1$ regardless of the value of $a_{p}$, and
$\big(\frac{a}{c_{2}}\big)=1$. Assume now $a\neq0$, and write
$\overline{\gamma\big(M_{2}(c)\big)}$ as
$\gamma\big(M_{2}(c)\big)^{a_{2}-1}\overline{\gamma\big(M_{2}(c)\big)}^{a_{2}}$.
By applying Lemma \ref{gammacomp} to $M_{p}(c)$ (whose discriminant has
cardinality $p^{mv_{p}(c)}\Delta_{M_{p}}$), we can replace
$\overline{\gamma\big(M_{p}(c)\big)}$ for odd $p$ and
$\overline{\gamma\big(M_{2}(c)\big)}^{a_{2}}$ for $p=2$ by
$\big(\frac{a_{p}}{p^{mv_{p}(c)}\Delta_{M_{p}}}\big)\overline{\gamma\big(M_{p}
(a_{p}c)\big)}$ and
$\big(\frac{2^{mv_{2}(c)}\Delta_{M_{2}}}{a_{2}}\big)\overline{\gamma\big(M_{2}
(a_{2}c)\big)}$ respectively. For odd $p$ we now cancel the two
$\big(\frac{a_{p}}{p^{v_{p}(c)}}\big)^{m}$ factors. On the other hand, for $p=2$
we can replace $\gamma\big(M_{2}(c)\big)^{a_{2}-1}$ by
$\gamma\big(M_{2}(c_{2})\big)^{a_{2}-1}$, as the Weil indices differ by a sign
and $a_{2}-1$ is even. Applying Lemma \ref{gammacomp} again and evaluating
$\gamma(M_{2})^{(a_{2}-1)(c_{2}-1)}$ transforms
$\gamma\big(M_{2}(c_{2})\big)^{a_{2}-1}$ into
$\gamma(M_{2})^{a_{2}-1}(-1)^{m\varepsilon(a_{2})\varepsilon(c_{2})}$. For
either odd or even $p$ we now decompose
$\overline{\gamma\big(M_{p}(a_{p}c)\big)}$ as the product of the Weil indices of
the Jordan components of $M_{p}(a_{p}c)$, which cancels with $\omega_{p}$
leaving the asserted product over $q$ not dividing $c$. It only remains to check
that the part with $q||c$ cancels as well. Here we have
$\overline{\gamma\big((p^{2v_{p}(c)})^{\kappa n}_{/t_{c}/II}(a_{p}c_{p})\big)}$
for any $p$, together with the term $\zeta_{8}^{a_{2}c_{2}t_{c}}$ for $p=2$.
Since the power of $p$ is even in this component, Proposition \ref{gammap} shows
that this complex conjugate Weil index is 1 for odd $p$ and cancels with
$\zeta_{8}^{a_{2}c_{2}t_{c}}$ from Lemma \ref{xcnorm} for $p=2$. This completes
the proof for the case where $c\neq0$.

For $c=0$ we have $D_{M}^{0*}=\{0\}$ with $\frac{\beta_{0}^{2}}{2}=0$,
$\Delta_{M,0}=\Delta_{M}$, and $a=d=\pm1$, so that the asserted formula becomes
$\prod_{p}\xi_{p}\cdot\mathbf{e}\big(\pm
b\frac{\gamma^{2}}{2}\big)e_{\pm\gamma}$. As $\mathbf{r}_{0}(A)$ comes with no
root of unity and the representations
$\rho_{M_{\mathbb{Q}_{p}}/\mathbb{Q}_{p}}\circ\iota$ and
$\rho_{M_{\mathbb{Q}_{2}}/\mathbb{Q}_{2}} \circ i$ yield the asserted
coefficients $\xi_{p}$ (recall the value of $\gamma(M_{p})^{2}$ for odd $p$ and
$a=-1$), this completes the proof of the theorem.
\end{proof}

In order to avoid the choice of a branch of $\sqrt{j(A,\tau)}$, we may replace
$\theta$ by the sign of $\Re\sqrt{j(A,\tau)}$ if $c\neq0$. If $c=0$ then
$\theta\sqrt{j(A,\tau)}$ is a constant $\delta\in\{\theta,-i\theta\}$ (by the
choice of branch of the square root), and as $\xi_{p}=\gamma(M_{p})^{1-a}$ for
odd $p$ in this case, the Weil reciprocity law implies that $\prod_{p}\xi_{p}$
coincides with $\delta^{-sgn(M)}$ for all the 4 cases of $\delta$. For the
principal branch of $\sqrt{j(A,\tau)}$ this agrees with the numbers denoted
$\xi_{\pm1,0}$ in \cite{[Str]}.

The reader who wishes to compare our results with those of \cite{[Sche]} and
\cite{[Str]} must be warned that \cite{[Sche]} works with the complex conjugate
representation. Moreover, our conventions for the Legendre symbols are
different, and in fact our roots of unity $\xi_{p}$ does not coincide with
theirs if $c\neq0$. Indeed, the ratio between our $\xi_{p}$ and the
corresponding coefficient of \cite{[Sche]} and \cite{[Str]} is
$\big(\frac{-1}{\Delta_{M_{p}}}\big)$ for odd $p$ and
$\overline{\gamma(M_{2})}^{2}$ for $p=2$. However, the global factor $i^{-sgnM}$
of \cite{[Str]} covers precisely for all these differences (the Weil reciprocity
law again), so that our final results do agree. Lemma \ref{gammacomp} can be
useful when one verifies the details of this comparison.

\medskip

Let $\widetilde{N}$ be the least common multiple of the orders of all the
elements of $D_{M}$. The level $N$ is $\widetilde{N}$ if the Jordan component
$(2^{v_{2}(\widetilde{N})})^{\kappa n}_{t/II}$ of $M_{2}$ has index $II$, and
equals $2\widetilde{N}$ otherwise. The latter case occurs only if
$\widetilde{N}$ is already even. Now, we have $\Delta_{M}^{c}=\{0\}$ if and only
if $N|c$, a case in which we have
$\rho_{M}\big(A,\theta\sqrt{j(A,\tau)}\big)e_{\gamma}=\varphi(A,\theta\sqrt{j(A,
\tau)})\mathbf{e}\bigg(bd\frac{\gamma^{2}}{2}\bigg)e_{d\gamma}$, where $\varphi$
is the product $\prod_{p}\xi_{p}$. Moreover, the product appearing in the
definition of $\xi_{p}$ in Theorem \ref{final} is empty in this case. In
addition, Lemma \ref{pfin} implies that either $\Delta_{M_{p}}=1$ or $a_{p}=a$
(coprimality). $\varphi\big(A,\theta\sqrt{j(A,\tau)}\big)$ therefore equals
\begin{equation}
\theta^{m}\bigg(\frac{a}{c_{2}}\bigg)^{m}(-1)^{m\varepsilon(a_{2})\varepsilon(c_
{2})}\bigg(\frac{2^{v_{2}(c)}}{a_{2}}\bigg)^{m}\bigg(\frac{\Delta_{M_{2}}}{a_{2}
}\bigg)\gamma(M_{2})^{a_{2}-1}\bigg(\frac{a}{\Delta_{M,2}}\bigg) \label{char}
\end{equation}
wherever $c\neq0$ (the case $a=0$ does not appear unless $N=1$ and $\rho_{M}$ is
trivial). This coefficient is just $\delta^{-sgn(M)}$ if $c=0$ and
$\theta\sqrt{j(A,\tau)}=\delta$. Examining the action on $e_{0}$ shows that the
map $\varphi$ is a character of the inverse image of $\Gamma_{0}(N)$ in
$Mp_{2}(\mathbb{Z})$. Let $\Gamma$ denote the subgroup of $\Gamma_{0}(N)$
defined by the congruences $N|b$ and $a \equiv d\equiv1(\mathrm{mod\
}\widetilde{N})$ (it contains $\Gamma(N)$ as a subgroup of index
$\frac{N}{\widetilde{N}}$). As the condition $\Delta_{M}^{c}=\{0\}$ is necessary
for an element of $Mp_{2}(\mathbb{Z})$ to be in $\ker\rho_{M}$, we find that
$\ker\rho_{M}$ consists of those elements of $\ker\varphi$ lying over $\Gamma$.
Explicitly, we get
\begin{prop}
The kernel of $\rho_{M}$ is a normal subgroup of $Mp_{2}(\mathbb{Z})$ which lies
over $\Gamma$, except for a few cases in which $\Gamma(N)\lneq\Gamma$ and it
lies over $\Gamma(N)$. These cases are $(i)$ $2||\widetilde{N}$ and
$\gamma(M_{2})^{2}\neq1$ (which always holds for odd $m$), and $(ii)$ $m$ is
even, $4||\widetilde{N}$, and $v_{2}(\Delta_{M})$ is odd. This kernel is a
double cover of $\Gamma$ or $\Gamma(N)$ respectively if $m$ is even, and it is a
lift of $\Gamma$ or of $\Gamma(N)$ if $m$ is odd. \label{kernel}
\end{prop}

In particular, Proposition \ref{kernel} implies the factoring of $\rho_{M}$
through (a double cover of) $SL_{2}(\mathbb{Z}/N\mathbb{Z})$. For the proof,
note that only the parts
$\big(\frac{\Delta_{M_{2}}}{a_{2}}\big)\gamma(M_{2})^{a_{2}-1}$ of Eq.
\eqref{char} require consideration. Furthermore, these are trivial wherever
$8|\widetilde{N}$, leaving very few possible forms for $M_{2}$ which we must
check. One first verifies that this product is 1 for elements of $\Gamma(N)$
(Eq. \eqref{oddmod8} is needed for the case $4||N$, and we need the Weil
reciprocity law for the case $c=0$). It remains to treat the cases where
$N=2\widetilde{N}$, the matrix in question is in $\Gamma\setminus\Gamma(N)$, and
either $2||\widetilde{N}$ and $4||N$ or $4||\widetilde{N}$ and $8||N$. A similar
case by case check completes the verification. We remark that if $m$ is odd then
$\ker\rho_{M}$ always lies over a subgroup of $\Gamma(4)$, and the lift is
always the restriction of $(i^{-1}\circ\iota)\cdot\psi^{v_{2}(\Delta_{M})+1}$,
where $\psi:\Gamma(4)\to\{\pm1\}$ takes $\binom{a\ \ b}{c\ \ d}$ to
$\big(\frac{2}{a}\big)$

\section{Odd Lattices and Further Generalizations \label{OddGen}}

In this Section we consider the changes we have to introduce if we take the
lattice $M$ to be odd rather than even, and describe briefly further possible
generalizations.

\medskip

The formulae in \cite{[Sn]} do not assume that the lattice is even. Instead, let
$\Gamma_{odd}$ denote the group $\Gamma(2) \cup S\Gamma(2)$ (or equivalently
$\Gamma_{0}^{0}(2) \cup S\Gamma_{0}^{0}(2)$) consisting of those matrices
$\binom{a\ \ b}{c\ \ d} \in SL_{2}(\mathbb{Z})$ in which both $ab$ and $cd$, or
equivalently both $ac$ and $bd$, are even. Condition (1.21) in \cite{[Sn]} means
that if the lattice is odd then the matrix lies in $\Gamma_{odd}$. We denote the
inverse image of $\Gamma_{odd}$ in $Mp_{2}(\mathbb{Z})$ by
$\widetilde{\Gamma}_{odd}$. It is generated by $T^{2}$ and $S$, and has index 3
in $Mp_{2}(\mathbb{Z})$ (just like $\Gamma_{odd}$ in $SL_{2}(\mathbb{Z})$). Now,
if $M$ is an odd lattice then $D_{M}$ does not carry a quadratic form, and the
construction from Section \ref{LatWRep} does not work. However, a careful
investigation of our proof shows that in this case the process of Sections
\ref{Local} and \ref{Main} yields a representation of $\Gamma_{odd}$, defined by
the same formulae from Theorem \ref{final}. Before we remark on this, observe
that in a more general setting the distinction is finer than just ``even'' and
``odd'' lattices: Given a lattice $M$ over a any ring $R$, define $I$ to be the
ideal generated by $x^{2}$ for all $x \in M$, and let $J=(2R:I)$ be the ideal
containing all those elements $r \in R$ such that $2|rs$ for any $s \in I$. The
ideal $J$ can be any ideal between $R$ (for even lattices) and $2R$ (the case of
``purely odd'' lattices). The corresponding subgroup of $SL_{2}(R)$ is
$\Gamma_{odd}^{J}=\Gamma_{0}^{0}(J) \cup S\Gamma_{0}^{0}(J)$, consisting of
those matrices in which $ab$ and $cd$ are in $J$. In the case $R=\mathcal{O}$
(with uniformizer $\pi$) considered in Section \ref{Local}, $I$ can be any
non-zero ideal $\pi^{t}\mathcal{O}$ in $\mathcal{O}$ (with some
$t\in\mathbb{N}$), and then $J=\pi^{\max\{0,v(2)-t\}}\mathcal{O}$. Having this
said, the reader who carries out the arguments from Sections \ref{Local} and
\ref{Main} should be aware that if $c \not\in J$ (which may happen only if
$2\not\in\mathcal{O}^{*}$) then $D_{M}^{c*}$ is not well-defined, as the subset
of $M_{\mathbb{K}}$ defined by the same property is no longer contained in
$D_{M}$ (if $c\in1+J$ then this yields the \emph{shadow} of $M$). However, for
matrices in $\Gamma_{odd}^{J}$ this happens only for $c\in\mathcal{O}^{*}$,
where $D_{M}^{c}=D_{M}$ and Lemma \ref{DMc*} is trivial. The fact that $a \in J$
in this case makes all the arguments (including Theorem \ref{M/cM}) work if we
choose $\tilde{x}_{c}$ to be 0, except for the factor from Lemma \ref{xcnorm}
being replaced by 1.

Our argument now proves the assertion in the first part of Proposition 1.6 of
\cite{[Sn]} in full generality. Now, for an odd lattice the inverse image of the
group $\Gamma$ from Proposition \ref{kernel} is contained in
$\Gamma_{0}^{0}(N)\subseteq\Gamma_{0}^{0}(2)\subseteq\Gamma_{odd}$. Let
$\widetilde{\Gamma}$ be the index 2 subgroup of $\Gamma(N)$ in which the
diagonal entries are congruent to 1 modulo $2N$. If the unimodular component in
a Jordan decomposition of the odd lattice $M_{2}$ is $1^{\kappa n}_{t}$ then the
previous paragraph and considerations like the one proving Proposition
\ref{kernel} establish
\begin{thm}
For an odd lattice $M$ the formulae from Theorem \ref{final} define a Weil
representation $\rho_{M}$ of $\Gamma_{odd}$ on $\mathbb{C}[D_{M}]$, except that
for odd $c$ we replace $\Delta_{M}^{c*}$ by $\Delta_{M}^{c}$ and add a
multiplier of $\zeta_{8}^{-a_{2}c_{2}t}$ to $\xi_{2}$. The kernel of $\rho_{M}$
lies over the group $\Gamma$ described before Proposition \ref{kernel}, except
for the following cases: $(i)$ $4||\widetilde{N}=\frac{N}{2}$, $m$ is even, and
$v_{2}(\Delta_{M})$ is odd; $(ii)$ $2||\widetilde{N}=\frac{N}{2}$ and $m$ is
odd; $(iii)$ $2||\widetilde{N}=\frac{N}{2}$, $m$ and $v_{2}(\Delta_{M})$ are
even, and $\gamma(M_{2})^{2}\neq1$; $(iv)$ $4||\widetilde{N}=N$, $m$ is even,
and $v_{2}(\Delta_{M})$ is odd; $(v)$ $2||N$ and $\gamma(M_{2})^{2}\neq1$;
$(vi)$ $2||\widetilde{N}=\frac{N}{2}$, $m$ is even, and $v_{2}(\Delta_{M})$ is
odd; $(vii)$ $\widetilde{N}$ is odd and $\gamma(M_{2})=1$. In cases
$(i)$--$(iii)$ the kernel of $\rho_{M}$ lies over $\Gamma(N)\lneq\Gamma$. In
cases $(iv)$ and $(v)$ it lies over $\widetilde{\Gamma}$. In case $(vi)$ it lies
over one of the two index 2 subgroups of $\Gamma$ which contain
$\widetilde{\Gamma}$ and do not equal $\Gamma(N)$. Finally, in case $(vii)$ it
is the inverse image of $\Gamma(\widetilde{N})\cap\Gamma_{odd}$, which strictly
contains $\Gamma$. In any case this kernel is a double cover of that group if
$m$ is even and it is a lift of it if $m$ is odd. \label{oddfin}
\end{thm}

The additional multiplier for $\xi_{2}$ where $c$ is odd appears in Theorem
\ref{oddfin} since the factor from Lemma \ref{xcnorm} no longer cancels the
corresponding Weil index. Observe that the case $N=2\widetilde{N}$ for odd
$\widetilde{N}$, in which $\Gamma=\Gamma(N)$, is now allowed. However, elements
of $\Gamma(\widetilde{N})\cap\Gamma_{odd}$ with odd $c$ (for which we have to
put the additional factor in $\xi_{2}$ in Eq. \eqref{char}) must also be
considered, since $\Delta_{M}^{c*}$ is replaced by $\Delta_{M}^{c}$ in the
formula for odd $c$. Note that in the cases $(iv)$--$(vi)$ of Theorem
\ref{oddfin} $\rho_{M}$ does not factor through the image of
$\widetilde{\Gamma}_{odd}$ in the double cover of
$SL_{2}(\mathbb{Z}/N\mathbb{Z})$. If $m$ is odd then the lift is again
$(i^{-1}\circ\iota)\cdot\psi^{v_{2}(\Delta_{M})+1}$, except in case $(v)$ of
Theorem \ref{oddfin} where the group is not contained in $\Gamma(4)$. We remark
that for odd lattices with odd $\Delta_{M}$, Lemma \ref{pfin} breaks down for
$p=2$ since $N$ is even. In addition, $\rho_{M_{2}}(S)$ acts as the scalar
$\overline{\gamma(M_{2})}=\zeta_{8}^{-t}$, which may or may not be trivial. As
for theta functions of odd lattices, the $V_{\rho_{M}}$-valued generalized theta
function from \cite{[B1]} is also modular, but now with respect to
$\widetilde{\Gamma}_{odd}$, as the argument from Theorem 4.1 of that reference
(with $T^{2}$ in place of $T$) shows. This extends the well-known modularity
property of the classical theta function
$\theta(\tau)=\sum_{n}\mathbf{e}\big(\frac{n^{2}\tau}{2}\big)$, which is a
special case of this more general function.

\medskip

The generality of the results of Section \ref{Local} suggests that it may be
possible to extend Theorem \ref{final} to lattices over the rings of integers in
other number fields. This method avoids examining the structure of
$Mp_{2}(\mathcal{O})$ for such a ring $\mathcal{O}$ (which may be complicated).
However, to carry on this task, we need to choose the characters on the local
fields properly: Recall that we have used both the product formula
$\mathbf{e}(x)=\prod_{p}\chi_{p}(x)$ and the fact that $\ker\rho_{p}$ was
precisely $\mathbb{Z}_{p}$ for every $p$. Finding characters for another number
field $\mathbb{F}$ which satisfy both properties is not so easy (since every
number field other that $\mathbb{Q}$ has a non-trivial discriminant). However,
one way to overcome this difficulty is to use the canonical choice of composing
$\chi_{p}$ with the trace to $\mathbb{Q}_{p}$ but allowing the bilinear forms of
$\mathcal{O}$-lattices to take values in the inverse different (and defining
$M^{*}$ accordingly). We remark that our results extend to the case of global
fields which are function fields (of characteristic $p\neq2$), where every
lattice is even and $\rho_{M}$ is always a (Weil) representation of
$SL_{2}(\mathcal{O})$.

In order to illustrate Theorem \ref{oddfin}, we give the formulae for the Weil
representation corresponding to an odd \emph{unimodular} lattice. In this case
$\rho_{M_{p}}$ is trivial for any odd $p$, $M_{2}\cong1^{\kappa m}_{t}$, and
$\gamma(M_{2})=\zeta_{8}^{sgn(M)}$ (so that
$t=sgn(M)\in\mathbb{Z}/8\mathbb{Z}$). The coefficients
$\frac{\sqrt{\Delta_{M,c}}}{\sqrt{\Delta_{M}}}$ and the roots of unity
depending on $\beta$ and $\gamma$ in the formula from Theorem \ref{final}
reduce to 1, and the representation is just the character
$\prod_{p}\xi_{p}=\xi_{2}$. Note that $m \equiv t(\mathrm{mod\ }2)$, so that we
may replace every instance of $m$ in the formula (which is always an exponent of
$\pm1$) by $t$. We therefore obtain the $t$th power of the character $\xi$ which
sends $\big(A,\theta\sqrt{j(A,\tau)}\big)\in\widetilde{\Gamma}_{odd}$ with
$A=\binom{a\ \ b}{c\ \ d}$ to $\theta\zeta_{8}^{1-a}$ if $c=0$,
$\frac{\theta}{\zeta_{8}^{c}}$ if $a=0$,
$\theta\big(\frac{a}{c_{2}}\big)(-1)^{\varepsilon(a_{2})\varepsilon(c_{2})}
\big(\frac{2^{v_{2}(c)}}{a_{2}}\big)\zeta_{8}^{a_{2}-1}$ when $ac\neq0$ and $c$
is even, and
$\theta\big(\frac{a}{c_{2}}\big)(-1)^{\varepsilon(a_{2})\varepsilon(c_{2})}
\big(\frac{2^{v_{2}(c)}}{a_{2}}\big)\zeta_{8}^{a_{2}-1-a_{2}c_{2}}$ when
$ac\neq0$ and $c$ is odd. $\xi$ has exponent 8 (hence its $t$th power is
well-defined), and we recall that for an even lattice we have $t=0$ and
$\xi^{t}$ is trivial. We remark that most of the non-trivial powers of $\xi$ are
not restrictions of characters of $Mp_{2}(\mathbb{Z})$. This is so, since the
Abelianization of $Mp_{2}(\mathbb{Z})$ is cyclic of order 24, in which $S$
coincides with $T^{-3}$ and has order 8 (this follows directly from the relation
$(ST)^{3}=S^{2}$), and its character group is generated by the character
$\varepsilon$ appearing in the transformation formula for the Dedekind $\eta$
function (indeed, $\varepsilon(T)=\mathbf{e}\big(\frac{1}{24}\big)$). As
$\xi(T^{2})=1$, only the trivial character of $Mp_{2}(\mathbb{Z})$ and the 12th
power of $\varepsilon$ (which restricts to $\xi^{4}$) extend powers of $\xi$.

\medskip

A more difficult problem arises from the fact that we have used the image of
$Mp_{2}(\mathbb{Z})$ in the (unique) infinite place of $\mathbb{Q}$, while the
technicalities of the double cover were pushed to the (unique) even place of
$\mathbb{Q}$. This has no immediate generalization, especially to the totally
complex case, where the Weil representation factors through $SL_{2}$ in every
complex place while we expect the global Weil representation of
$Mp_{2}(\mathcal{O})$ not to factor through $SL_{2}(\mathcal{O})$ for lattices
of odd rank. A possible solution to this problem lies in the next paragraph. We
leave the more detailed analysis for future work.

Another interesting question may be to extend our results for obtaining the
explicit formulae for the action of larger subgroups of (covers of) the
symplectic group $Sp(M \times M)$ preserving the anti-symmetrization of the
bilinear form on $M \times M$ for an even $\mathbb{Z}$-lattice $M$. It is
reasonable to expect that the compact subgroup $Sp(M_{p} \times M_{p})$ of
$Sp_{\mathbb{Q}_{p}}(M_{\mathbb{Q}_{p}} \times M_{\mathbb{Q}_{p}})$ still
preserves the finite-dimensional space of $\mathcal{S}(M_{\mathbb{Q}_{p}})$
from Section \ref{Decom} and that it acts trivially if the prime $p$ satisfies
the conditions of Lemma \ref{pfin}. It should then be possible to combine the
methods of Section \ref{Meta} and the ideas of \cite{[Ra]} in order to obtain
the metaplectic double covers in our terminology, which hopefully splits again
over $Sp(M_{p} \times M_{p})$ for odd $p$. We thus conjecture that a similar
tensor product argument yields a representation of a double cover of $Sp(M
\times M)$, whose restriction to $Mp_{2}(\mathbb{Z})$ is our $\rho_{M}$.
Applying the methods of Sections \ref{Local} and \ref{Main} to this case,
combined with the formulae of \cite{[Ra]}, one may obtain the general formulae
for this representation. We remark that the paper \cite{[Sh]} gives, using theta
functions again, some formulae for a similar action of elements in a symplectic
generalization of $\Gamma_{0}(N)$ on such spaces (see Propositions 3b.1 and 3b.2
of that reference). In particular, if $\mathcal{O}$ is the ring of integers of a
global field and $M$ is an $\mathcal{O}$-lattice (with a bilinear form taking
values in the inverse different, say), determining the action of the subgroup
$Mp_{2}(\mathcal{O})$ of the double cover of $Sp(M \times M)$ should be a
feasible task (since again Eqs. \eqref{Wc=0} and \eqref{Wcinv} suffice). All
this, however, is the suggested subject for future work.

\noindent\textsc{Fachbereich Mathematik, AG 5, Technische Universit\"{a}t
Darmstadt, Schlossgartenstrasse 7, D-64289, Darmstadt, Germany}

\noindent E-mail address: zemel@mathematik.tu-darmstadt.de

\end{document}